\documentclass[12pt]{article}

\def\cA{\mathcal{A}}

\makeatother

\usepackage{xcolor}
\definecolor{bp}{rgb}{0.2, 0.2, 0.6}

\usepackage[margin=3cm]{geometry}
\usepackage[normalem]{ulem}
\usepackage{amsthm, bbm}
\usepackage{amsmath, amssymb}
\usepackage{mathrsfs}
\usepackage{hyperref}
\usepackage{graphicx, psfrag, epstopdf}
\usepackage{float}
\UseRawInputEncoding

\theoremstyle{definition}
\newtheorem{definition}{Definition}
\newtheorem{theorem}{Theorem}
\newtheorem{lemma}{Lemma}[section]
\newtheorem{sublemma}{Sublemma}[section]
\newtheorem{corollary}{Corollary}[section]
\newtheorem{proposition}{Proposition}[section]

\newtheorem{question}{Question}[section]
\newtheorem{remark}{Remark}[section]

\numberwithin{equation}{section}

\definecolor{OliveGreen}{rgb}{0,0.6,0}

\def\R{\mathbb{R}}
\def\Z{\mathbb{Z}}
\def\T{\mathbb{T}}
\def\N{\mathbb{N}}
\def\a{{\bf a}}

\theoremstyle{definition}
\newtheorem*{inner}{\innerheader}
\newcommand{\innerheader}{}
\newenvironment{namedtheorem}[1]
 {\renewcommand\innerheader{#1}\begin{inner}}
 {\end{inner}}

\title{A non-mixing Arnold flow on a surface}
\author{Bassam Fayad, Adam Kanigowski, Rigoberto Zelada}

\begin{document}
\maketitle
\begin{abstract}

We construct a smooth area preserving flow  on a genus $2$ surface with exactly one open uniquely ergodic component, that is asymmetrically bounded by separatrices of non-degenerate saddles and that is       nevertheless not mixing. 
 \end{abstract}

\tableofcontents

\section{Introduction}

\subsection{Non-mixing Arnold flows} 

Area preserving flows on surfaces form the most basic class of continuous time conservative dynamics.  These flows are often called {\em multi-valued, or locally, Hamiltonian flows}, following the terminology introduced by S.~P.~Novikov \cite{novikov82},  who 
studied them in his elaboration of a Morse theory of pseudoperiodic manifolds.  

Smooth conservative surface flows preserve by definition a smooth area form, hence their flow lines 
form a foliation induced by the symplectic dual of a closed 1-form, which is locally  given by the exterior derivative of a multi-valued Hamiltonian function. Besides their intrinsic importance in topology and geometry, the study of these flows is motivated by solid state physics \cite{novikov95}. They are a special case of  foliations induced by closed 1-forms on a compact manifold $M$, the study of which was thoroughly developed in the last decades since foundational works by Novikov, Arnold, Zorich, Dynnikov, and others. We refer to \cite{FFK} and \cite{chaika} and the survey \cite{CUsurvey} for accounts regarding the literature concerned with the statistical behavior of 
multi-valued Hamiltonian flows.


{It follows from independent works of  Mayer \cite{May} , Levitt \cite{Lev}, and Zorich \cite{Zor}, that each smooth area-preserving flow can be decomposed into finitely many integrable components and quasi-minimal components: an integrable component is a subsurface (possibly with boundary) on which all orbits are closed and periodic (topologically these components are discs or cylinders); quasi-minimal components (there are not more than $g$ of them) are subsurfaces (possibly with boundary) on which the flow is quasi-minimal in the sense that all trajectories that do not converge to the singularities are dense.  Moreover each component is bounded by separatrices which correspond to orbits whose forward or backward trajectory hits a singularity of the flow.

 Moreover, on each  minimal component, the flow can be represented as a special flow above interval exchange transformation (IET) where the discontinuities of the IET correspond to orbits that meet singularities of the flow. The ceiling function of the special flow is smooth away from the discontinuities where it has  infinite asymptotic values. 
The asymptotes depend on the type of the singularities. When the singularities are non-degenerate these asymptotes are logarithmic. 

 In~\cite{arnold}, Arnol'd showed considered multi-valued Hamiltonian flows with non-degenerate saddle points on the torus that have a phase portrait that decomposes into elliptic islands (topological disks bounded by saddle connections and filled up by periodic orbits) and one open uniquely ergodic component. The roof function for these examples has typically asymmetric logarithmic singularities since the coefficient in front of the logarithm is twice as big on one side of the singularity as the one on the other side, due to the existence of homoclinic loops (see Figure  \ref{FigSaddle}). This also happens typically for flows on higher genus surfaces provided that there are homoclinic loops. We will call Arnold flow (or Arnold component) the flow on the open minimal component in the (logarithmic) asymmetric case.}

\begin{definition} An Arnold flow is a special flow above an IET and under a roof function that is smooth except at the discontinuities of the IET where it has  logarithmic asymptotes, and such that the sum of the coefficients in front of the increasing logarithmic asymptotes is different from the sum of the coefficients in front of the decreasing logarithmic asymptotes. 
\end{definition}

\begin{figure}[htb]
 \centering
 \resizebox{!}{5.4cm} {\includegraphics[angle=1]{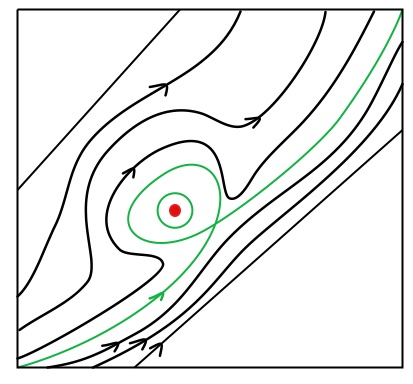}}
  \caption{\small Multivalued Hamiltonian flow. When there is a loop, starting at the same distance from the separatrix, the orbits passing to the left of the saddle 
spend approximately twice longer time comparing 
to the orbits passing to the right of the saddle.}
 \label{FigSaddle}
\end{figure}

Arnol'd conjectured that  in this case the flow will in general be mixing on the open ergodic component. This conjecture is now known to hold for typical Arnold flows due to a series of results:

1) K.~Khanin and Ya. G.~Sinai \cite{KS} gave the first mixing examples of special flows above a class of circle rotations and with a ceiling function having two asymmetric logarithmic asymptotes.
 
2) A. Kochergin~\cite{Koc4,Koc9} extended the result of \cite{KS} to include all irrational rotation angles and all asymmetric logarithmic ceiling functions (with any finite number of asymptotes). 

  { 
3) C. Ulcigrai \cite{CU1} showed that mixing holds for Arnold flows above a full measure set of interval exchange transformations and with the roof function having one (asymmetric) singularity.

4) D. Ravotti \cite{Rav} extended the results of \cite{CU1} for Arnold flows above  a full measure set of interval exchange transformations and the roof function having finitely many singularities (in fact Ravotti obtained quantitative mixing).
}

Our goal here is  to give an example of an Arnold flow that is not mixing. Our main result is:
\begin{theorem}\label{mejn}There exists a smooth area preserving flow on a genus $2$ surface which has four integrable and one uniquely ergodic Arnold component that is not mixing. 
\end{theorem}

\subsection{Statistical behavior of area preserving flows on surfaces. The global picture}

Let us consider our result within the bigger picture of statistical behavior of conservative surface flows.
 When statistical properties on an open ergodic component are studied, there are, depending on the singularities of the flow, three main different scenarios.

As explained above, the flows can be viewed as special flows above IETs with ceiling functions smooth away from the discontinuities of the IET and  
 having infinite asymptotic values at the discontinuities.  The three {main cases studied in the literature are}  $(i)$ and  $(ii.a)$ and  $(ii.b)$ below.\footnote{We believe that in the case of analytic locally Hamiltonian flows degenerate singularities always produce power singularities. If this is the case then any analytic locally Hamiltonian flow belongs to either one of  $(i)$,  $(ii.a)$ or $(ii.b)$.}

\medskip 

$(i)$ The ceiling function has at least one power-like asymptote. This is the case where the flow has at least one degenerate singularity.

$(ii)$ All asymptotes are logarithmic. This holds when all the singularities of the flow are non-degenerate.

$(ii.a)$ Symmetric case: The sum of the coefficients in front of the increasing logarithmic asymptotes sums up to the same amount as the coefficients in front of the decreasing logarithmic asymptotes. 

$(ii.b)$ Asymmetric case: Arnold flows.

In case $(i)$, Kochergin showed that mixing always holds, \cite{Koc2}. Fraction-polynomial rate of mixing is  typically expected and was proved above a full measure set of rotations \cite{Fa_SMF}. Countable Lebesgue spectrum is also typically expected for such flows and was proved above a full measure set of rotations \cite{FFK}. The latter two results are not yet investigated above general IETs. The idea behind mixing is that the shear caused by any power-like asymptote is sufficient to produce mixing and cannot be compensated by any other asymptote.

In case  $(ii.a)$, absence of mixing is the typical outcome, as proved by Kochergin for special flows above irrational rotations \cite{Koc2,Koc4}, and by Ulcigrai for typical IETs \cite{CU3}.


The idea behind the absence of mixing in Kochergin's proof is that for symmetric logarithmic singularities, a Denjoy-Koksma like property (DK property) holds above irrational rotations that prevents mixing of the special flow.  Denjoy-Koksma times are integers for which the Birkhoff sums have a bounded oscillation around the mean value on all or on a positive measure proportion of the base (see for example  the discussion around 
DK property in \cite{DF}). 

In higher genus, the situation is more delicate because 
of polynomial deviations of Birkhoff sums from the mean \cite{Zorich3}, \cite{forni}. However, Ulcigrai \cite{CU3}
 proved that, despite these deviations, for almost all IET's there are still sufficient cancellations to prevent mixing.  A different cancellation mechanism was found slightly earlier by Scheglov \cite{scheglov}, but that was special to the genus $2$ context.

 However,  as proven by Chaika and Wright \cite{chaika}, these cancellations do not happen for all {uniquely ergodic} IETs, because for some IETs the lack of uniformity in the convergence of Birkhoff sums due to polynomial deviations is important and makes mixing possible.  Indeed, the examples of \cite{chaika} have symmetric logarithmic singularities but are mixing.

\medskip 

In case $(ii.b)$, mixing holds typically. As mentioned earlier, this was proved by { Ravotti \cite{Rav}} and Ulcigrai \cite{CU1}  after earlier works by Khanin and Sinai \cite{KS}, and Kochergin ~\cite{Koc4,Koc9}.

Our contribution here is to show that mixing in case  $(ii.b)$ may fail the same way as {non-mixing in case}  $(ii.a)$ may fail. As a matter of fact, our example for absence of mixing under an asymmetric ceiling functions is directly inspired by the example  of mixing under a symmetric ceiling functions of \cite{chaika}.

The mechanism behind absence of mixing in our examples is dual to that of \cite{chaika}: There, the asymmetry in the IET dynamics disrupts the cancellation in the shear and gives mixing. In our example, the asymmetry in the IET dynamics  compensates exactly the asymmetry in the ceiling function and ends up yielding a Denjoy-Koksma like property on a positive part of the space which overrules mixing.
This exact compensation requires delicate estimates along a subsequence of time that are not necessary for \cite{chaika}. On the other hand, and unlike \cite{chaika}, we do not need to estimate the Birkhoff sums along all times but only along a subsequence to prove that mixing fails.

\medskip 
\subsection{The explicit construction}\label{sec:smooth}
We describe now the concrete example in Theorem \ref{mejn}. The construction is similar to the one in \cite{chaika}. We start with the vertical translation flow on a genus  two surface obtained by glueing two identical tori each sheared by $\alpha$ and that are glued along an identical slit as in Figure \ref{fig_slit'}. On the figure, the slit is the interval delimited by two red crosses. The first return map $T$ to the union of the two horizontal circles has discontinuities at $-\alpha\times\{j\}$ and $-\alpha+|J|\times \{j\}$, $j=1,2$ (two orange intervals in figure \ref{fig_slit'}). It is a $\Z^2$ extension over the interval $J$ of the irrational rotation by $\alpha$. The green loops add an artificial (symmetric) singularity which is not a discontinuity of $T$.  The vertical flow is $C^\infty$ outside two cone points (the identified red crosses) at which the cone angle is $4\pi$, the flow is singular at those points. However by a time change around those points (a slow down) one can now get a globally smooth flow (defined everywhere on the surface). The two cone points become fixed points of the time-changed flow. This procedure is described in detail in \cite{CFr} (see also \cite{chaika}).
 The first return map to the base is still $T$ and the return time function is smooth except above the pre-images of the two fixed points at which it blows up logarithmically.  To produce asymmetric logarithmic singularities we glue in two identical asymmetric loops on the pair of orbits starting at $-\alpha+|J|$ (see the pair of blue lines in Figure \ref{fig_slit'} and the two blue loops in Figure \ref{fig_slit2}). This will produce asymmetric logarithmic singularities oriented identically at $(-\alpha+|J|,i)$, $i=0,1.$  Moreover we glue in two extra identical asymmetric loops on opposite sides of the orbit of the point $0$ (see the green line in Figure \ref{fig_slit'} and the green orbit with two opposite loops in Figure \ref{fig_slit2}). This will produce an extra symmetric singularity at the point $0\times \{1\}$ (as the asymmetric contributions of each green loop cancel out perfectly to yield a symmetric one).  We refer to Section 7 of \cite{CFr} to see how these asymmetric loops can be glued in in a smooth way. 

Summarizing,  this construction gives a flow on a genus $2$ surface with $4$  integrable components inside the loops and exactly one open ergodic component represented by a special flow above  $T:\T\times \Z^2\to \T\times \Z^2$  and under a ceiling function $f$ that has $4$ singularities at $(-\alpha,j)$ and $(-\alpha+|J|,j)$, $j=0,1$.  The ceiling function $f$ has symmetric logarithmic singularities over $(|J|-\alpha,j)$ of the form $-\log x$ (see formulas \eqref{eq:g}--\eqref{eq:f} with $x_1=|J|-\alpha$),  asymmetric singularities over $(-\alpha,j)$  of the form $-2\log x$ and $-3\log x$ (see formulas \eqref{eq:g}--\eqref{eq:f} with $x_0=-\alpha$) and an extra symmetric singularity over $0\times\{1\}$ of the form $A\log x$ (see formula \eqref{eq:h}).

\begin{figure} 
\centering
 \resizebox{!}{5.4cm} {\includegraphics[angle=1]{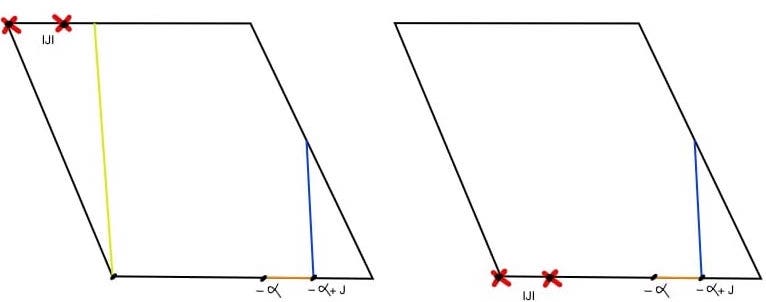}}
\caption{\small Genus 2 surface -- two tori glued along a slit. The two identical tori are sheared by the number $\alpha$. The vertical flow on the surface is smooth except a pair of two cone points (pairs of identified red crosses) at which the cone angle is $4\pi$. The first return map $T$ acts on the union of the two copies of $[0,1)$. Its discontinuities are located at $(-\alpha,i)$ and $(-\alpha+|J|,i)$, where $i=0,1$. The green and the pair of blue lines correspond to orbits along which we glue in smoothly asymmetric loops.} \label{fig_slit'}
\end{figure}

\begin{figure} \label{fig_slit2}
 \resizebox{!}{5.4cm} {\includegraphics[angle=1]{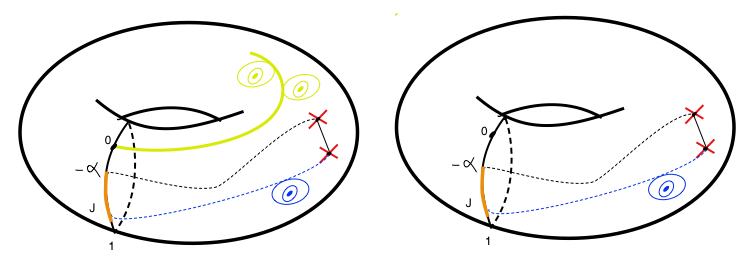}}
\caption{\small After the slowdown, the (identified) pairs of red crosses represent the two symmetric saddles. Their preimages along the dotted curves give discontinuities of $T$. There is an extra asymmetric loop on a pair of discontinuities (in blue). Finally there are two asymmetric loops placed symmetrically over $0\times \{1\}$ (in green). They asymmetric contributions cancel out and hence they produce an extra symmetric singularity at $0\times \{1\}$.}  \label{fig_slit2}
\end{figure}
The use of a base interval exchange transformation of the form of $T$ is similar to the one used by Chaika and Wright \cite{chaika}, that was in turn inspired by \cite{veech,katok,sataev}. 
Observe that by a result of Kochergin (Theorem 2 in \cite{Kocc}), if follows that any smooth flow with (strongly) asymmetric singularities on the $\T^2$ is mixing and so our example is optimal in terms of the genus of the surface. 
The return time above the points $z_1$ and $z_2$ has identical asymmetric logarithmic asymptotes whose global contribution is therefore  asymmetric logarithmic.  
The use of a base interval exchange transformation of the form of $T$ is similar to the one used by Chaika and Wright \cite{chaika}, that was in turn inspired by \cite{veech,katok,sataev}. 
Observe that by a result of Kochergin (Theorem 2 in \cite{Kocc}), if follows that any smooth flow with (strongly) asymmetric singularities on the $\T^2$ is mixing and so our example is optimal in terms of the genus of the surface. 
 Note that we glue two identical saddle loops (blue loops in Figure  \ref{fig_slit2}). It then follows that the study of the contribution to ergodic sums over $T$ of the blue parts reduces to studying ergodic sums over the rotation by $\alpha$ and for one blue loop. By a more careful analysis of the ergodic sums of one of the blue parts over $T$ it seems possible to glue in just one blue loop. This would result in a smooth flow on a genus $2$ surface with only $3$ saddle loops. The two green asymmetric loops are introduced to obtain a symmetric singularity over  $0\times \{1\}$. It seems that by passing to the $\Z^2\times \Z^2$ extension, analogously to the construction in Chaika-Wright, \cite{chaika} one can in fact produce a symmetric singularity over $\{0\}\times \{1\}$ which comes from a simple saddle and not  two (opposite) saddle loops.  This would give a smooth flow on a surface of genus $5$ with only $1$ saddle loop and such that on the ergodic component the flow has one asymmetric singularity and is  not mixing. It is therefore natural (in parallel to \cite{chaika}) to ask the following:
\begin{question}
Does there exist a smooth flow on a genus two surface with only non-degenerated saddles and only two invariant components (one integrable and one minimal)?
\end{question}

 \section{A special flows above a $\Z_2$ extension of a circle rotation} \label{sec2}

We now give a formal definition of the special flow with base dynamics $T$ and ceiling function $f$ that was described in Section \ref{sec:smooth}. It will be characterized by the choice of $\alpha_0$, $J$ and $A$. The map $T$ is a $\Z_2$ extension of a circle rotation by an irrational $\alpha_0\in (0,1)$ with \textit{ad-hoc} Diophantine properties.  
 In order to prove the non-mixing property of $(T^f_t)$, the special flow over $T$ and under $f$, we will need the Diophantine properties of $\alpha_0=\alpha$ to guarantee the existence of  a sequence $(r_n)_{n\in\N}$ along which we have fine control up to bounded oscillations of ergodic sums of the roof function $S_{r_n}(T,f)(x)$, for a positive measure set of points $x\in \T\times \{0,1\}$. For this, the construction of $\alpha_0$ needs some more specification than in the work of  Chaika and Wright \cite{chaika}.\\

   Let $R=R_{\alpha_0}:\T\to \T$, 
 $R(x)=x+\alpha_0$ and let $(q_n)_{n\in\N}$ denote the sequence of denominators of $\alpha_0$. For an interval $J\subset \T$ define 
 \begin{equation}\label{eq:skew}
  T=T_{\alpha_0,J}:\T\times \Z_2\to \T\times \Z_2, \;\;\;\;\;\;\;\ T(x,j)=\Big(x+\alpha_0, j+\chi_J(R(x))\Big).
 \end{equation}
 The map $T$ preserves the Lebesgue measure on $\T\times \Z_2$.
 Notice that (without loss of generality) the map $T$ has $4$ discontinuities:
 \begin{align}\label{eq:discontinuitiesOfT}
     z_{1}=(x_0,0), && z_{2}=(x_0,1), && z_{3}=(x_1,0), &&\text{ and } z_{4}=(x_1,1),
 \end{align}
 where $x_0=-\alpha_0$ and $x_1=|J|-\alpha_0$. 
 We will define the roof function $f$ to have logarithmic singularities at discontinuities of $T$ with an extra discontinuity at 
 \begin{equation}\label{eq:ExtraDiscontinuity}
     z_0=(0,1)\in\mathbb T\times\Z_2.
 \end{equation}
 In fact, $f$ will have asymmetric logarithmic singularities at $z_{1}$ and $z_{2}$, symmetric logarithmic singularities at $z_{3}$ and $z_{4}$, and a symmetric logarithmic singularity over $z_0$. In particular, $f$ will have asymmetric logarithmic singularities.\\
 Before defining the function $f$ we would like to introduce some notation.
 For any $r\in \R$, we let 
 $$\log^+(r)=\begin{cases}
\log(r),\text{ if }r>0,\\
0,\text{ if }r\leq 0
 \end{cases}
 $$
 and set $\|r\|$ to denote the distance from $r$ to the closest integer, meaning that
$$\|r\|=\inf_{n\in\Z}|r-n|.$$
 We define $g:\T\times\Z_2\rightarrow\R_+$ by 
 \begin{equation}\label{eq:g}
 g(x,j)=1+ 2\Big|\log(\|x-x_0\|)\Big|+\Big|\log^+(x_0-x)\Big| +
 \Big|\log(\|x-x_1\|)\Big|,
 \end{equation}
where, while computing $\log^+(x_0-x)$, we identify $x_0$ with the irrational $1-\alpha_0$ and $x\in\T$ with its unique representative in $[0,1)$. For any given $A>1$, we also let $h=h_A:\T\times\Z_2\rightarrow\R_+$ be defined by 
 \begin{equation}\label{eq:h}
 h(x,j)=A\cdot\delta_1(j)\cdot \Big|\log(\|x\|)\Big|.
 \end{equation}
 We now define $f=f_{A,\a_0,J}:\T\times\Z_2\rightarrow \R_+$ by 
 \begin{equation}\label{eq:f}
 f(x,j)=g(x,j)+h(x,j).
 \end{equation}
 Note that (a) $f$ has asymmetric logarithmic singularities over $z_1$ and $z_2$, (b) symmetric singularities over $z_3$ and $z_4$,  and (c) another symmetric singularity at  $\{0\}\times \{1\}\in \T\times \Z_2$.
 \begin{remark}
     Observe that the domains of $f'$, $g'$, and $h'$ are not the same as those of $f$, $g$, and $h$ respectively. Throughout this work we will replace  $f'$, $g'$, and $h'$, by right continuous functions having the same domain as $f$, $g$, or $h$. So, for example, in Proposition \ref{prop:h_1=phi} we will let 
     $$h_1'(x,j)=\chi_{\{1\}}(j)\frac{\chi_{[1/2,1)}(x)-\chi_{(0,1/2)}(x)}{\|x\|}.$$
     (Note that, unlike $h_1'$ above,  the derivative of $h_1$ is not defined at $(1/2,1)$.) We will employ similar extensions for $f''$, $g''$, and $h''$.
 \end{remark}
 In Section \ref{sec:smooth} we explain how the corresponding special flow corresponds to a smooth flow on a genus $2$ surface.
Our main theorem that implies Theorem \ref{mejn}, is:
\begin{namedtheorem}{Theorem A}\hypertarget{thm:main2}{}
There exists $\alpha_0\in \T$, $J\subset \T$ and $A>1$ such that the special flow build over $T=T_{\alpha_0,J}$ and under the roof function $f=f_A$ is not mixing. Furthermore, the IET $T$ is uniquely ergodic.
\end{namedtheorem}
In the sequel we will denote by $T^f$  the special flow defined on Theorem  \hyperlink{thm:main2}{A}.

Before we finish this section, we make some observations about the space  $\T\times\Z_2$. We will denote the normalized Lebesgue measure on $\mathbb T$ by $\lambda'$ and the normalized Lebesgue measure on 
$\T\times \Z_2$ by $\lambda$. In other words, when we identify $\Z_2$ with $\{0,1\}$,
$$\lambda=\frac{1}{2}(\lambda'\otimes \delta_{0}+\lambda'\otimes \delta_{1}).$$ 
Note that the map $(x,y)\mapsto \|x-y\|$ defined on $\T\times \T$ defines a metric generating the topology of $\T$. We will  let $d$ be the metric defined on $\T\times\Z_2$ by 
\begin{equation}\label{eq:MetricOfTxZ_2}
d((x,i),(y,j))=\|x-y\|+\delta_i(j).
\end{equation}
We remark that $d$ generates the product topology of $\T\times \Z_2$.\\
The following proposition collects some of the basic properties of the metric $d$. We omit the proof. 
\begin{proposition}\label{prop:BasicPropertiesOfd}
Let $T:\T\times\Z_2\rightarrow\T\times\Z_2$ be as in \eqref{eq:skew} and let $i,j\in\Z_2$. For any $x,y\in \T$,
\begin{enumerate}
\item [(i)] If $d((x,i),(y,j))<1$, then $i=j$. 
\item [(ii)] $d((x,j),(y,j))=\|x-y\|$.
\item [(iii)] If $T(x,j),T(y,j)\in \T\times\{i\}$, then  $d((x,j),(y,j))=d(T(x,j),T(y,j))$.
\end{enumerate}
\end{proposition}
\section{Criterion for absence of mixing}

\subsection{Continued fractions and Denjoy-Koksma inequality} \label{BasicNotationSection}

For any irrational number $\alpha\in [0,1]$  we denote by $[a_0^\alpha;a_1^\alpha,...]$ the continued fraction expansion of $\alpha$ and let  $(p_n^{\alpha})_{n\in\N}$ and $(q_n^\alpha)_{n\in\N}$ be the sequences defined recursively by
\begin{equation}\label{eq:NumeratorRecursiveDefn}
p_{n+1}^{\alpha}=a_{n+1}^{\alpha}p_n^\alpha+p_{n-1}^\alpha
\end{equation}
and 
\begin{equation}\label{eq:DenominatorRecursiveDefn}
q_{n+1}^\alpha=a_{n+1}^\alpha q_n^\alpha+q_{n-1}^\alpha
\end{equation}
where $p_{-1}^\alpha=1$, $p_0^\alpha=a_0$ and $q_{-1}^\alpha=0$, $q_0^\alpha=1$. 
For any $n\in\N$, it holds that 
\begin{equation}\label{eq:BestApproxA_ell}
\frac{1}{q_n^\alpha(q_n^\alpha+q_{n+1}^\alpha)}<\left|\alpha-\frac{p_n^\alpha}{q_n^\alpha}\right|<\frac{1}{q_n^\alpha q_{n+1}^\alpha}.
\end{equation}
Furthermore, when $n$ is even, 
$$0<\alpha-\frac{p_n^\alpha}{q_n^\alpha}$$
and when $n$ is odd,
$$\alpha-\frac{p_n^\alpha}{q_n^\alpha}<0.$$
When it is clear from context, we will write $[a_0;a_1,...]$, $(q_n)_{n\in\N}$, and $(p_n)_{n\in\N}$ instead of $[a_0^{\alpha};a_1^{\alpha},...]$, $(q_n^{\alpha})_{n\in\N}$, and  $(p_n^{\alpha})_{n\in\N}$, respectively.\\
The following result will be used to prove Theorem \hyperlink{thm:main2}{A}.
\begin{proposition}\label{prop:DistributionOfIrrationalsWithSimilarContinuedFraction}
Let $n\in\N$ and let the irrational numbers $\alpha,\beta\in (0,1)$ be such that 
$$q_n^{\alpha}=q_{n}^\beta\text{ and }p_n^{\alpha}=p_n^{\beta}.$$
Set $q_n=q_n^\alpha=q_n^\beta$.
Then for any $k\in\{1,...,q_{n}-1\}$ there exists a $c_k\in\{0,...,q_{n}-1\}$ such that  
$$k\alpha,k\beta\mod 1\in \left(\frac{c_k}{q_n},\frac{c_k+1}{q_n}\right).$$
Furthermore, for each $\ell\in\{0,...,q_n-1\}$, there is a $k\in\{0,...,q_n-1\}$ with 
$$k\alpha\mod 1\in [\frac{\ell}{q_n},\frac{\ell+1}{q_n}].$$
\end{proposition} 
\begin{proof}
Set $p_n=p_n^\alpha=p_n^\beta$. When $n$ is even $\alpha,\beta>\frac{p_n}{q_n}$ and when $n$ is odd, $\alpha,\beta<\frac{p_n}{q_n}$. Thus, by \eqref{eq:BestApproxA_ell},  we either have 
$$0<k\alpha-\frac{kp_n}{q_n},k\beta-\frac{kp_n}{q_n}<\frac{1}{q_n}$$
for each $k\in\{1,...,q_n-1\}$ or 
$$0<\frac{kp_n}{q_n}-k\alpha,\frac{kp_n}{q_n}-k\beta<\frac{1}{q_n}.$$
for each $k\in\{1,...,q_n-1\}$. Setting $c_k\equiv kp_n\mod q_n$ in the former case and $c_k\equiv (kp_n-1)\mod q_n$ in the latter, we see that the first claim holds.\\
To see that the second claim holds, recall that $p_n$ and $q_n$ have no non-trivial common divisors.
\end{proof}

We now record for future use Denjoy-Koksma inequality. For any set $X$, any $x\in X$, any function $F:X\rightarrow \mathbb C$, any map $Q:X\rightarrow X$, and any $n\in\N$, we define the $n$-th ergodic sum of $f$ along the orbit of $Q$ at $x$ by
$$S_n(Q,F)(x)=\sum_{j=0}^{n -1}F(Q^jx).$$
So, in particular, when for some $\alpha\in (0,1)$, $X=\T$ and $Q=R_{\alpha}$,
$$S_n(R_\alpha,F)(x)=\sum_{j=0}^{n-1}F(x+j\alpha\mod 1).$$
\begin{namedtheorem}{Denjoy-Koksma inequality}
    Let $\alpha$ be an irrational number with denominator sequence $(q_n)_{n\in\N}$ and let $F:\mathbb T\rightarrow \R$ be a function of bounded variation. For any $n\in\N$ and any $x\in\T$,
    $$\left|S_{q_n}(R_\alpha,F)(x)-q_n\int_\T F\text{d}\lambda'\right|\leq\text{Var}(F),$$
    where $\lambda'$ denotes the normalized Lebesgue measure on $\T$ and $\text{Var}(F)$ the total variation of $F$.
\end{namedtheorem}

\subsection{Criterion for absence of mixing for special flows}
The following condition on absence of mixing of special flows built over general transformations $(G,X,\kappa)$ was introduced by Kochergin, \cite{Koch}. 
\begin{proposition}[Absence of Mixing criterion, Theorem 1 in \cite{Koch}]\label{prop:abs} Let $(G^\phi)$ be the special flow build over a map $(G,X,\kappa)$ and under a function $\phi:X\to \R_+$. Assume there exists a constant $M>0$, an increasing sequence of integers $(r_k)_{k\in\N}$ and a sequence of sets $E_k\subset X$, $k\in\N$,  satisfying the following:
\begin{enumerate}
\item[A1.]  $\inf_{k\in\N} \lambda (E_k)>0$ and 
$$\lim_{k\rightarrow\infty}\left(\sup_{z\in E_k}d(z,G^{r_k}z) \right)=0.$$
\item[A2.] for any $y_1,y_2\in E_k$, $|S_{r_k}(G,\phi)(y_1)-S_{r_k}(G,\phi)(y_2)|<M$;
\end{enumerate}
then $(G^\phi)$ is {\em not} mixing.
\end{proposition}
In fact the formulation of condition A1. is slightly different in \cite{Koch} but then in Remark 2 in \cite{Koch} the author explains that the condition A1. implies the conditions in Theorem 1. We will use the above criterion for special flows over the skew product map $T:\T\times\Z_2\rightarrow\T\times\Z_2$. In fact we will state another criterion which is based on Proposition \ref{prop:abs} and which  is adapted to the flow $T^f$ defined in Section \ref{sec2}. Recall that  $T=T_{\alpha_0,J}$ and $f$ were defined in \eqref{eq:skew} and \eqref{eq:f}, respectively, that $z_1,z_2,z_3,z_4$ denote the discontinuities of $T$, and that $f$ possess an additional discontinuity at $z_0=(0,1)\in\T\times\Z_2$. 
\begin{proposition}\label{prop:crit} Let $(q_n)_{n\in\N}$ be the sequence of denominators of $\alpha_0$. Assume there exist $C,c>0$,  a sequence $(t_k)_{k\in\N}$ in $\N$, and  a sequence of points $(y_k,j_k)_{k\in\N}$ in $\T\times \Z_2$ such that  for every $k\in\N$:
\begin{enumerate}
\item[B1.] $$\min_{0\leq j<q_{t_k},1\leq i\leq 4} d\Big(T^j(y_k,j_k),z_i\Big)\geq \frac{2c}{q_{t_k}};$$
\item[B2.] $d(T^{q_{t_k}}(y_k,j_k),(y_k,j_k))=o((q_{t_k}\log q_{t_k})^{-1})$;
\item[B3.] $|S_{q_{t_k}}(T,f')(y_k,j_k)|<C\cdot q_{t_k}$;
\item[B4.] $|S_{q_{t_k}}(T,f'')(x,j_k)|<C\cdot q_{t_k}^2$ for any $x\in\T$ with $\|x-y_k\|\leq \frac{c}{q_{t_k}}$;
\item[B5.] $|S_{j}(T,f')(x,j_k)|<C\cdot q_{t_k}\log q_{t_k}$ for any $x\in\T$ with $\|x-y_k\|\leq \frac{c}{q_{t_k}}$ and any $j<q_{t_k}$.
\end{enumerate}
then the special flow $(T^f)$ is not mixing.
\end{proposition}
\begin{proof}[Proof of Proposition \ref{prop:crit}] For each $k\in\N$, let $I_k=[-\frac{c}{2q_{t_k}}+y_k,y_k+\frac{c}{2q_{t_k}}]$ and let $E_k:=\bigcup_{j=0}^{q_{t_k}-1}T^j(I_k\times \{j_k\})$. We claim that $\inf_k \lambda(E_k)>0$.\\
Indeed, for any distinct $i,j\in\{0,...,q_{t_k}-1\}$, 
$$\|(j-i)\alpha_0\|\geq \|q_{t_k-1}\alpha_0\|>\frac{1}{2q_{t_k}}$$
and, hence, 
\begin{equation}\label{eq:MeasureOfEk}
\lambda(E_k)\geq \frac{1}{2}\lambda'(\bigcup_{j=0}^{q_{t_k}-1}(I_k+j\alpha))\geq\frac{1}{2}q_{t_k}\frac{\min\{1,c\}}{2q_{t_k}}\geq \frac{\min\{1,c\}}{4},
\end{equation}
proving the claim.\\

Notice now that  for each $j\in\{0,...,q_{t_k}-1\}$, $T^j(I_k\times\{j_k\})$ is connected and no discontinuity of $T$ lies in $T^j(I_k\times\{j_k\})$, i.e. $z_i\notin T^j(I_k\times\{j_k\})$ for $i=1,\ldots,4$. To see this, first note  that $I_k\times\{j_k\}$ is connected and that, by $B_1$, it contains no discontinuity of $T$ (Otherwise we would have that for some $i\in\{1,...,4\}$, $d((y_k,j_k),z_i)<\frac{c}{2q_{t_k}}$). Thus, $T(I_k\times\{j_k\})$ is connected and, by $B_1$ and Proposition \ref{prop:BasicPropertiesOfd} item (iii), $T(I_k\times\{j_k\})$ contains no discontinuity of $T$. Continuing in this way, we see that each of $T^2(I_k\times\{j_k\})$,....,$T^{q_{t_k}-1}(I_k\times\{j_k\})$ are connected and contain no discontinuity of $T$.\\
Set 
$$
H_k=(I_k\times\{j_k\})\cup T^{q_{t_k}}(I_k\times\{j_k\}).
$$
We now show that for $k\in\N$ large enough we have that for any  $j\in\{0,...,q_{t_k}-1\}$, $T^{j}H_k$ is connected and contains no discontinuity of $T$. To do this, first note that since $T^{q_{t_k}-1}(I_k\times\{j_k\})$ is connected and contains no discontinuity of $T$, 
$
T^{q_{t_k}}(I_k\times\{j_k\})
$
is also connected. By condition $B_2$,
$$\lim_{k\rightarrow\infty}q_{t_k}d(T^{q_{t_k}}(y_k,j_k),(y_k,j_k))=0.$$
Thus, taking $k\in\N$ large enough, we have that 
$$d(T^{q_{t_k}}(y_k,j_k),(y_k,j_k))<\frac{c}{2q_{t_k}}<1.$$
So, by Proposition \ref{prop:BasicPropertiesOfd} items (i) and (ii),
$(I_k\times\{j_k\})\cap (T^{q_{t_k}}(I_k\times\{j_k\})\neq \emptyset$ and, hence,
$H_k$
is also connected. It now follows from $B_1$, Proposition \ref{prop:BasicPropertiesOfd} item (ii), and  our choice of $k$ that $H_k$  contains no discontinuity of $T$. Note now that $TH_k$ is connected and that, by $B_1$, Proposition \ref{prop:BasicPropertiesOfd} item (ii), and our choice of $k$, it contains no discontinuity of $T$. Continuing in this way for $T^jH_k$, $j=2,...,q_{t_k}-1$, we see that the claim holds. \\
It now follows that  for $k\in\N$ large enough, any $j\in\{0,...,q_{t_k}-1\}$, and any $(x,j_x)\in T^j(I_k\times\{j_k\})$, 
\begin{equation}\label{eq:DistanceOnRecurrentTimes}
d(T^{q_{t_k}}(x,j_x),(x,j_x))=\|q_{t_k}\alpha_0\|=d(T^{q_{t_k}}(y_k,j_k),(y_k,j_k))=o\Big((q_{t_k}\log(q_{t_k}))^{-1}\Big),
\end{equation}
which, together with \eqref{eq:MeasureOfEk}, imply  that condition $A.1$ in  in Proposition \ref{prop:abs} holds with $(r_k)_{k\in\N}=(q_{t_k})_{k\in\N}$.\\

 Let $\vec{y}\in T^{j}(I_k\times\{j_k\})$ for some $j\in\{0,...,q_{t_k}-1\}$. We use the cocycle identity to write 
\begin{multline}\label{eq:CocycleDecomp}
 S_{q_{t_k}}(T,f)(\vec{y})-S_{q_{t_k}}(T,f)(y_k,j_k)=\Big(S_{q_{t_k}}(T,f)(T^{-j}\vec{y})-S_{q_{t_k}}(T,f)(y_k,j_k)\Big)-\\
 \Big(S_j(T,f)(T^{-j}\vec{y})-S_j(T,f)(T^{q_{t_k}-j}\vec{y})\Big)
 \end{multline}
The bound on both of the summands in the right-hand side of \eqref{eq:CocycleDecomp} is similar. Observe that for $k\in\N$ large enough and for each $j\in\{0,...,q_{t_k}-1\}$, $T^jH_k$ is connected and contains no discontinuity of $T$ and that, when $j<q_{t_k}-1$, the restriction of $T$ to 
$T^jH_k$
is an isometry. Viewing $S_{q_{t_k}}(T,f)$ as a function on $H_k$ and  
using Taylor expansion, we can write
$$
|S_{q_{t_k}}(T,f)(T^{-j}(\vec{y}))-S_{q_{t_k}}(T,f)(y_k,j_k)|\leq |S_{q_{t_k}}(T,f')(y_k,j_k)|\cdot |I_k|+ |S_{q_{t_k}}(T,f'')(\theta)||I_k|^2,
 $$
 for some $\theta\in H_k$. By B3, $|S_{q_{t_k}}(T,f')(y_k,j_k)|\cdot |I_k|<c\cdot C$ and, by B4,  $|S_{q_{t_k}}(T,f'')(\theta)|=O(q_{t_k}^2)$. This bounds the first summand in the right-hand side of \eqref{eq:CocycleDecomp}. For the second summand,  we  write (using the mean value theorem)
 \begin{multline*}
 |S_j(T,f)(T^{-j}\vec{y})-S_j(T,f)(T^{q_{t_k}-j}\vec{y})|\\
 =|S_{j}(T,f')(\theta)|d(T^{-j}\vec{y},T^{q_{t_k}-j}\vec{y})=O(q_{t_k}\log q_{t_k})\cdot o((q_{t_k}\log q_{t_k})^{-1})=o(1), 
 \end{multline*}
 by condition $B5$ and \eqref{eq:DistanceOnRecurrentTimes}. This finishes the proof. 
\end{proof}

 We will use Proposition \ref{prop:crit} to prove Theorem \hyperlink{thm:main2}{A}. We will construct an $\alpha_0$ with sequence of denominators $(q_n)_{n\in\N}$ and an interval $J\subseteq \mathbb T$ such that for some increasing sequence $(t_k)_{k\in\N}$ in $\N$, the sequence $q_{t_k}$, $k\in\N$, satisfies conditions B1-B5 in Proposition \ref{prop:crit}.\\
 Conditions B1, B4, and B5 will be easily satisfied just by general properties of distribution of the orbit of $\alpha_0$. Condition B2. will be  guaranteed by certain Diophantine assumption on $\alpha_0$: $q_{t_k}\log q_{t_k}=o(q_{t_k+1})$. Condition B3. is by far the most difficult and requires most of the work.
  
\section{Ergodic sums over rotations for functions with asymmetric logarithmic singularities} 
 We now state some general lemmas on distribution of orbits of an irrational rotation by $\alpha_0$ on $\T$. We will denote by $(q_n)_{n\in\N}$ the sequence of denominators of $\alpha_0$ and for any $x\in\T$, we will let $Rx=x+\alpha_0\mod 1$. The main tool for establishing this results is  Denjoy-Koksma inequality (See Section \ref{BasicNotationSection}).  Recall that $g$ is defined by \eqref{eq:g} and $x_0$, $x_1$ by \eqref{eq:discontinuitiesOfT}.  We define the function $\gamma:\mathbb T\rightarrow \R$ by $\gamma(x)=g(x,0)$. More explicitly, for any $x\in\mathbb T$,
$$
\gamma(x)=
 1+ 2\Big|\log(\|x-x_0\|)\Big|+\Big|\log^+(x_0-x)\Big| +
 \Big|\log(\|x-x_1\|)\Big|
$$
 \begin{lemma}\label{lem:DK} For any $n\in\N$ and any $x\in\T$ with 
 $$\min_{0\leq s<q_n,0\leq i\leq 1}\|x+s\alpha_0-x_i\|>0,$$
 we have 
 \begin{equation}\label{eq:ClosestReturnBoundForG'PartI}
 \Big|S_{q_n}(R,\gamma')(x)-q_n\log q_n \Big|=O\Big([\min_{0\leq s<q_n,0\leq i\leq 1}\|x+s\alpha_0-x_i\|]^{-1}\Big).
 \end{equation}
Moreover, for any $j<q_n$,
\begin{equation}\label{eq:ClosestReturnBoundForG'PartII}
 |S_{j}(R,\gamma')(x)|\leq 7q_n\log q_n+O\Big([\min_{0\leq s<q_n,0\leq i\leq 1}\|x+s\alpha_0-x_i\|]^{-1}\Big).
 \end{equation}
 \end{lemma}
 \begin{proof}
 Adapting the proof of Lemma 4.3. in \cite{FaKa}, one obtains
 \begin{equation}\label{eq:L4.3}
 4q_n\geq S_{q_n}(R,-\frac{\chi_{[x_i,x_i+1/2)}(x)}{|x-x_i|})+q_n\log(q_n)\geq -(\min_{0\leq s<q_n}\|x+s\alpha_0-x_i\|)^{-1}-4q_n,
 \end{equation}
 \begin{equation}\label{eq:L4.3'}
 -4q_n\leq S_{q_n}(R,\frac{\chi_{[x_i-1/2,x_i)}(x)}{|x-x_i|})-q_n\log(q_n)\leq (\min_{0\leq s<q_n}\|x+s\alpha_0-x_i\|)^{-1}+4q_n,
 \end{equation}
 and
 \begin{multline}\label{eq:L4.3''}
 q_n\left(\log(\frac{2}{1-\alpha_0})-4\right)\leq S_{q_n}(R,\frac{\chi_{[0,1-\alpha_0)}(x)}{(1-\alpha_0)-x})-q_n\log(q_n)\\
 \leq (\min_{0\leq s<q_n}\|x+s\alpha_0-x_i\|)^{-1}+q_n\left(\log(\frac{2}{1-\alpha_0})+4\right),
 \end{multline}
 for every  $x\in \T$, $i\in\{0,1\}$, and $n\in\N$. Observe that, by Lemma \ref{prop:DistributionOfIrrationalsWithSimilarContinuedFraction}, 
 $$\min_{0\leq s<q_n,0\leq i\leq 1}\|x+s\alpha_0-x_i\|\leq \frac{1}{q_n}$$
 and that   $\gamma'$ can be expressed as the sum of five monotone functions $\gamma'_1,...,\gamma'_5$ each of which is of the form 
 $$\gamma'_t(x)=A_t\frac{\chi_{B_t}(x)}{|x-x_i|}$$
 where $A_t$ is a constant, $B_t$ is a subset of $\T$, and $\sum_{t=1}^5A_t=1$.  Thus, by \eqref{eq:L4.3}, \eqref{eq:L4.3'}, and \eqref{eq:L4.3''} one sees that \eqref{eq:ClosestReturnBoundForG'PartI} holds.\\
 To see that \eqref{eq:ClosestReturnBoundForG'PartII} holds, it is enough to pick $j<q_n$ and note that for each $t\in\{1,...,5\}$, $|S_{j}(R,\gamma_t')(x)|<|S_{q_n}(R,\gamma_t')(x)|$. We are done. 
 \end{proof}
 The following corollary is immediate:
 \begin{corollary}\label{cor:away}
Let $n\in\N$, let $x\in \T$, and let $c>0$. Suppose that
$$\min_{0\leq s<q_n,0\leq i\leq 1}\|x+s\alpha_0-x_i\|>\frac{c}{q_n}.$$
Then, 
$$
S_{q_n}(R,\gamma')(x)=q_n\log q_n+O(q_n).
$$
 \end{corollary}
We will also need:
  \begin{lemma}\label{lem:DK0} 
For any $n\in\N$ and any $x\in\T$ with 
 $$\min_{0\leq s< q_n,0\leq i\leq 1}\|x+s\alpha_0-x_i\|>0,$$
 we have 
 $$
 |S_{q_n}(R,\gamma'')(x)|=O\Big([\min_{0\leq s<q_n,0\leq i\leq 1}\|x+s\alpha_0-x_i\|]^{-2}\Big)
 $$
 \end{lemma}
 \begin{proof}
Note that 
$$\frac{d^2}{dx^2}\Big(-\log|x|\Big)=\frac{1}{x^2}$$
and, hence,
$$\gamma''(x)=\frac{2}{\|x-x_0\|^2}+\frac{\chi_{[0,1-\alpha_0)}}{|1-\alpha_0-x|^2}+\frac{1}{\|x-x_1\|^2}.$$ 
Thus, adapting the proof of Lemma 4.3. in \cite{FaKa}, we obtain
\begin{multline*}
|S_{q_n}(R,\gamma'')(x)|=O\Big(q_n^2\Big)+O\Big([\min_{0\leq s<q_n,0\leq i\leq 1}\|x+s\alpha_0-x_i\|]^{-2}\Big)\\
=O\Big([\min_{0\leq s<q_n,0\leq i\leq 1}\|x+s\alpha_0-x_i\|]^{-2}\Big).
\end{multline*}
 \end{proof}
 
 \section{Proof of Theorem A}\label{sec:ProofOFTheoremA}
Our goal in this section is to present the different results involved in proving  Theorem \hyperlink{thm:main2}{A} and then utilize them to prove that $T_f$ is non-mixing. We collect these results in Proposition \hyperlink{prop:PropertiesOFAlpha0}{B} below. Before stating Proposition {B}, we introduce various definitions which will allow us to reformulate  Theorem \hyperlink{thm:main2}{A} in more precise terms.\\

For a given irrational $\alpha$ with denominators sequence $(q_n)_{n\in\N}$ and a  strictly increasing sequence $(n_k)_{k\in\N}$ of even numbers in $\N$, let 
 \begin{equation}\label{eq:J_alpha}
 J_\alpha(n_k):=\Big[0, \sum_{k=1}^{+\infty}2\|q_{n_k}\alpha\|\mod 1\Big)
 \end{equation}
 and, for any $(x,j)\in\mathbb T\times\Z_2$, set 
 \begin{equation}\label{eq:skewAlpha}
 T(x,j)=T_{\alpha,n_k}(x,j)=(x+\alpha, j+\chi_{J_{\alpha}(n_k)}(x+\alpha)).
 \end{equation}
 We remark that \eqref{eq:J_alpha} and \eqref{eq:skewAlpha} are well defined since for any irrational $\alpha$ and any increasing sequence $(n_k)_{k\in\N}$ in $2\N$, $\sum_{s=1}^\infty\|q_{n_s}\alpha\|<\infty$.
\begin{namedtheorem}{Theorem A$^\prime$} \hypertarget{theorem:Main2'}{}
    There exists $\alpha_0\in \T$, an increasing sequence of even numbers $(n_k)_{k\in\N}$ in $\N$, and an $A>1$ such that the IET $T=T_{\alpha_0,n_k}$ and  $T^f$, the special flow build over $T$ and under the roof function $f=f_{A,{\alpha_0},J_{\alpha_0}(n_k)}$, satisfy the following:
    \begin{enumerate}
        \item [(i)] $T^f$ is not mixing.
        \item [(ii)] $T$ is uniquely ergodic. 
    \end{enumerate}
\end{namedtheorem}
To alleviate the notation, when there is no risk of ambiguity, we will simply use the notations $T$ and $f_A$ for $T_{\alpha_0,n_k}$ and $f_{A,{\alpha_0}, J_{\alpha_0}(n_k)}$.\\

 Item (ii) 
  is obtained in a fashion similar to that used in \cite{chaika} and will be proved in Proposition \ref{thm:UniqueErgodicity} of Section \ref{sec:ue}. We will now prove (i) with the help of  Proposition \hyperlink{prop:PropertiesOFAlpha0}{B} which in turn  will be proved in the next two sections. 
\begin{namedtheorem}{Proposition B}\hypertarget{prop:PropertiesOFAlpha0}{}
There exist  an irrational $\alpha_0=[0;a_1,a_2,....]$, an increasing sequence of even numbers $(n_k)_{k\in\N}$, an increasing sequence $(t_k)_{k\in\N}$ in $\N$, a sequence $(y_k)_{k\in\N}$ in $\mathbb T$, a sequence $(j_k)_{k\in\N}$ in $\{0,1\}$, and a constant $A>1$ such that, if we take  $T=T_{\alpha_0,n_k}$,
$f=f_{A,{\alpha_0},J_{\alpha_0}(n_k)}$, then the following holds with $(q_n)_{n\in\N}$ the sequence of  denominators of the best rational approximations of $\alpha_0$ 
\begin{enumerate}
    \item [($\alpha$.1)] For every $k\in\N$ large enough and $h=h_A$ as defined in \eqref{eq:h}, 
    $$|S_{q_{t_k}}(T,h')(y_k,j_k)+q_{t_k}\log(q_{t_k})|<O(q_{t_k}).$$
    \item [($\alpha$.2)] For every $k\in\N$ large enough,
    \begin{equation}\label{eq:alpha.2}
    \min_{0\leq j\leq q_{t_k}}\|y_k+j\alpha_0-|J_{\alpha_0}(n_k)|\|,\min_{0\leq j\leq q_{t_k}}\|y_k+j\alpha_{0}\|\geq  \frac{1}{16q_{t_k}}.
    \end{equation}
    \item [($\alpha$.3)] For every $k\in\N$ large enough, there exists a $z\in\T$ such that $T^{q_{t_k}}(y_k,j_k)=(z,j_k)$.
     \item [($\alpha$.4)] $\lim_{k\rightarrow\infty}\frac{\log(q_{t_k})}{a_{(t_k+1)}}=0$
    \item [($\alpha$.5)] For every $k\in\N$, $a_{n_{(2k-1)}+1}= 3$.
    \item [($\alpha$.6)]  $\lim_{k\rightarrow\infty}(\sum_{s=1}^kq_{n_s})/q_{n_{(k+1)}}=0$.
    \item [($\alpha$.7)] $\lim_{k\rightarrow\infty}a_{n_{2k}+1}=\infty$. 
    \item [($\alpha$.8)] $\lim_{k\rightarrow\infty}(q_{n_k}\sum_{s=k+1}^\infty\|q_{n_s}\alpha_0\|)=0.$
\end{enumerate}
\end{namedtheorem}
\begin{remark}
  The only properties of  $\alpha_0$ in  Proposition \hyperlink{prop:PropertiesOFAlpha0}{B} that we need to prove that $T^f$ is not mixing are conditions ($\alpha$.1)-($\alpha$.4). Conditions ($\alpha$.5)-($\alpha$.8) will be used in  Section \ref{sec:ue} to prove  the unique ergodicity of $T_{\alpha_0,n_k}$. 
\end{remark}
\begin{proof}[Proof of (i) of Theorem \hyperlink{thm:main2'}{A$^\prime$}]
Let $\alpha_0\in\T$, $(n_k)_{k\in\N}$ in $\N$, $(t_k)_{k\in\N}$ in $\N$, $(y_k)_{k\in\N}\in \T$, $(j_k)_{k\in\N}\in\Z_2$, and $A>1$ be as in the statement of Proposition \hyperlink{prop:PropertiesOFAlpha0}{B}.
 All we need to show is that there exist $C,c>0$ for which conditions $B_1$-$B_5$ in Proposition \ref{prop:crit} hold for sucifficiently large $k\in\N$, $T=T_{\alpha_0,n_k}$, and $f=f_{A,{\alpha_0},J_{\alpha_0}(n_k)}$ as defined in \eqref{eq:f}.\\
 We see that, for sufficiently large $k\in\N$, ($\alpha$.2) implies that 
 \begin{equation*}
     \min_{0\leq j<q_{t_k},1\leq i\leq 4} d\Big(T^j(y_k,j_k),z_i\Big)\geq 
      \min_{0\leq j\leq q_{t_k},0\leq i\leq 1}\|y_k+(j-1)\alpha_0-x_i\|\geq \frac{1}{16q_{t_k}},
 \end{equation*}
 which implies $B.1$ with $c=\frac{1}{32}$. By ($\alpha$.3) and (ii) of Proposition \ref{prop:BasicPropertiesOfd},  we obtain
 $$d((y_k,j_k),T^{q_{t_k}}(y_k,j_k))=\|y_k-(y_k+q_{t_k}\alpha_0)\|=\|q_{t_k}\alpha_0\|<\frac{1}{q_{t_k+1}}<\frac{1}{a_{t_k+1}q_{t_k}}$$
 for $k\in\N$ large enough. 
So, by ($\alpha$.4), $B.2$ holds.\\
Since $f=g+h$,
\begin{multline*}
|S_{q_{t_k}}(T,f')(y_k,j_k)|\leq|S_{q_{t_k}}(T,g')(y_k,j_k)-q_{t_k}\log(q_{t_k})|+|S_{q_{t_k}}(T,h')(y_k,j_k)+q_{t_k}\log(q_{t_k})|\\
= |S_{q_{t_k}}(R_{\alpha_0},\gamma')(y_k)-q_{t_k}\log(q_{t_k})|+|S_{q_{t_k}}(T,h')(y_k,j_k)+q_{t_k}\log(q_{t_k})|,
\end{multline*}
where $\gamma(x)=g(x,0)=g(x,1)$ for each $x\in\T$. 
So, by ($\alpha$.1), ($\alpha$.2), and Corollary \ref{cor:away}, $B.3$ holds (for $k\in\N$ large enough).\\
To prove $B.4$ and $B.5$, take $k\in\N$ large enough to ensure that \eqref{eq:alpha.2} holds and let $x\in\T$ be such that $\|x-y_k\|\leq \frac{1}{32q_{t_k}}$. It follows that 
  \begin{equation}\label{eq:xAwayFromSingularitiesAlpha2}
  \min_{0\leq j\leq q_{t_k}}\|x+j\alpha_0-|J_{\alpha_0}(n_k)|\|,\min_{0\leq j\leq q_{t_k}}\|x+j\alpha_{0}\|
    \geq  \frac{1}{32q_{t_k}}
\end{equation}
and, hence, by Lemma  \ref{lem:DK0}, 
$$|S_{q_{t_k}}(T,g'')(x,j_k)|=|S_{q_{t_k}}(R_{\alpha_0},\gamma'')(x)|=|S_{q_{t_k}}(R_{\alpha_0},\gamma'')(x)|\leq O(q_{t_k}^2)$$
and, by Lemma \ref{lem:DK}, for any $j\in\{0,...,q_{t_k}-1\}$, 
$$|S_{j}(T,g')(x,j_k)|=|S_j(R_{\alpha_0},\gamma')(x)|\leq O(q_{t_k}\log(q_{t_k}))$$
Observe that \eqref{eq:xAwayFromSingularitiesAlpha2} implies that
$$\min_{0\leq j\leq q_{t_k}}\|x+j\alpha_{0}\|
    \geq  \frac{1}{32q_{t_k}}$$
for every $x\in\T$ with $\|x-y_k\|\leq \frac{1}{32q_{t_k}}$ and $k\in\N$ large enough. Thus, by arguing as in the proofs of Lemmas \ref{lem:DK} and  \ref{lem:DK0}, we  obtain
$$|S_{q_{t_k}}(T,h'')(x,j_k)|\leq|S_{q_{t_k}}(T,\frac{A}{\|x\|^2})(x,j_k)|=|S_{q_{t_k}}(R_{\alpha_0},\frac{A}{\|x\|^2})(x)|=O(q_{t_k}^2)$$
and, for each $j<q_{t_k}$, 
$$|S_j(T,h')(x,j_k)|\leq |S_j(T,\frac{A}{\|x\|})(x,j_k)|=O(q_{t_k}\log(q_{t_k})).$$
Thus, 
$$|S_{q_{t_k}}(T,f'')(x,j_k)|\leq |S_{q_{t_k}}(T,g'')(x,j_k)|+|S_{q_{t_k}}(T,h'')(x,j_k)|= O(q_{t_k}^2)$$
and 
$$|S_{j}(T,f')(x,j_k)|\leq |S_{j}(T,g')(x,j_k)|+|S_{j}(T,h')(x,j_k)|=O(q_{t_k}\log(q_{t_k}))$$
proving that $B.4$ and $B.5$ hold.
\end{proof}
 \section{Construction of $\alpha_0$ and $J$ and  choice of the constant A. Proof of Proposition B}
 In this section we prove Proposition \hyperlink{prop:PropertiesOFAlpha0}{B} with the help of Proposition \hyperlink{prop:unbalancedB}{C}, which is closely related with condition ($\alpha$.1) in Proposition  \hyperlink{prop:PropertiesOFAlpha0}{B}. To state Proposition \hyperlink{prop:unbalancedB}{C}, we first need to introduce non-minimal approximations of the map $T$ that will be helpful in the control of the Birkhoff sums of $f$ above $T$.
We remark that the use of these periodic approximations imitates that presented in \cite{chaika}. 

\subsection{Birkhoff sums above non-minimal approximations of the map $T$}
Let $\alpha\in (0,1)$ be an irrational number with  denominators sequence $(q_n^{\alpha})_{n\in\N}=(q_n)_{n\in\N}$ and let $(n_k)_{k\in\N}$ be an increasing sequence in $2\N$. For any $s\in\N$ and any $(x,j)\in \T\times\Z_2$, we define
\begin{equation}\label{eq:SthSkew}
T_{\alpha,s}(x,j)=(x+\alpha, j+\chi_{J^s_\alpha(n_k)}(x+\alpha)),
\end{equation}
where $J^s_\alpha(n_k)=[0,2\sum_{k=1}^s\|q_{n_k}\alpha\|\mod 1)$.\\
The transformations $(T_{\alpha,s})$ are non-minimal approximations of $T_\alpha=T_{\alpha,n_k}$ in the following sense: for every large enough $s\in \N$ and every $K\in \N$,
 \begin{equation}\label{eq:saz}
 \text{ if } x\notin \bigcup_{w=1}^K R^{-w}_\alpha\Big(J_\alpha(n_k)\setminus J^s_{\alpha}(n_k)\Big),\;\;\; \text{ then }\;\;\;\; T^K_\alpha(x,j)=T^K_{\alpha,s}(x,j).
 \end{equation}
For each $m\in\N$, we let 
\begin{equation}\label{eq:DefnJ'_m}
J'_{\alpha,m}=J'_{\alpha,m}(n_k)=\Big[2\sum_{s=0}^{m-1}q_{n_s}\alpha,2\sum_{s=0}^{m-1}q_{n_s}\alpha+q_{n_m}\alpha\Big)\times\Z_2,
\end{equation} 
where $q_{n_0}=0$ and for any $x\in\T$, 
$$x\in\Big[2\sum_{s=0}^{m-1}q_{n_s}\alpha,2\sum_{s=0}^{m-1}q_{n_s}\alpha+q_{n_m}\alpha\Big)\subseteq \T$$
if and only if there is a representative $y\in\R$ of $x$ with $$2\sum_{s=0}^{m-1}q_{n_s}\alpha\leq y<2\sum_{s=0}^{m-1}q_{n_s}\alpha+\|q_{n_m}\alpha\|.$$
(We remark that for any $n\in\N$, $|q^\alpha_n\alpha-p^\alpha_n|<1/2$.)\\
We also  set 
  \begin{multline}\label{eq:DefnU_m}
  U_{\alpha,m}=U_{\alpha,m}(n_k)\\
  =\left(U_{\alpha,(m-1)}\setminus \bigcup_{i=0}^{q_{n_m}-1}T_{\alpha,(m-1)}^i(J'_{\alpha,m})\right)\cup\left(V_{\alpha,(m-1)}\cap\bigcup_{i=0}^{q_{n_m}-1}T_{\alpha,(m-1)}^i(J'_{\alpha,m})\right),
  \end{multline}
     and
  \begin{multline}\label{eq:DefnV_m}
  V_{\alpha,m}=V_{\alpha,m}(n_k)\\
 =\left(V_{\alpha,(m-1)}\setminus \bigcup_{i=0}^{q_{n_m}-1}T_{\alpha,(m-1)}^i(J'_{\alpha,m})\right)\cup\left(U_{\alpha,(m-1)}\cap\bigcup_{i=0}^{q_{n_m}-1}T_{\alpha,(m-1)}^i(J'_{\alpha,m})\right),
  \end{multline}
(where $U_{\alpha,0}=\mathbb T\times\{0\}$, $V_{\alpha,0}=\mathbb T\times\{1\}$, and $T_{\alpha,0}(x,j)=(x+\alpha,j)$). We remark that $U_{\alpha,m}$ and $V_{\alpha,m}$ form a partition of $\T\times\Z_2$ for each $m\in\N\cup\{0\}$ and, hence, for $m\in\N$,
\begin{equation}\label{eq:DefnOfU_m'}
U_{\alpha,m}
  =\left( U_{\alpha,(m-1)}\cup\bigcup_{i=0}^{q_{n_m}-1}T_{\alpha,(m-1)}^i(J'_{\alpha,m})\right) \setminus \left(U_{\alpha,(m-1)}\cap \bigcup_{i=0}^{q_{n_m}-1}T_{\alpha,(m-1)}^i(J'_{\alpha,m})\right),
\end{equation}
which in turn implies that 
\begin{equation}\label{eq:SymDiffOfU_m}
    U_{\alpha,m}\triangle U_{\alpha,(m-1)}=\bigcup_{i=0}^{q_{n_m}-1}T_{\alpha,(m-1)}^i(J'_{\alpha,m}).
\end{equation}

We record for future use the following facts about $T_{\alpha,m}$, $U_{\alpha,m}$, and $V_{\alpha,m}$.
\begin{lemma}\label{lem:BasicFactsAboutUalphaTalpha}
For any irrational $\alpha$ with denominator sequence $(q_n)_{n\in\N}$ and any increasing sequence $(n_k)_{k\in\N}$ in $2\N$,  the following statements hold for any $m\in\N$:
\begin{enumerate}
    \item [(i)] If $2\sum_{s=1}^m\|q_{n_s}\alpha\|<1$, then $U_{\alpha,m}$ and $V_{\alpha,m}$ are both $T_{\alpha,m}$-invariant. 
    \item [(ii)] $U_{\alpha,m}$ and $V_{\alpha,m}$ are disjoint unions of intervals of the form 
    \begin{equation}\label{eq:ShapeOfUandV}
    [L\alpha\mod 1,R\alpha\mod 1)\times\{j\},
    \end{equation}
    $L,R\in\{0,...,2\sum_{s=1}^mq_{n_s}-1\}$ and $j\in\{0,1\}$.\footnote{
     Note that in \eqref{eq:ShapeOfUandV} we make use of the identification  $0\equiv 1\mod 1$
    } 
   Furthermore, 
    the set of discontinuities of $\chi_{U_{\alpha,m}}$ and $\chi_{V_{\alpha,m}}$ are each equal to  
    $$\Delta_m=\{(k\alpha\mod 1,j)\,|\,k\in\{0,...,2\sum_{s=1}^mq_{n_s}-1\}\text{ and }j\in\{0,1\}\}.$$
    
    \item [(iii)]  Let $(x,j)\in\T\times\Z_2$. Then 
    \begin{equation}\label{eq:InvolutionCondition}
    (x,j)\in U_{\alpha,m},\text{ if and only if }(x,j+1\mod 2)\in V_{\alpha,m}.
    \end{equation}
    
    \item [(iv)] If $2\sum_{s=1}^{r}q_{n_s}<q_{n_r+1}$ for each $r\in\{1,...,m\}$, then 
\begin{equation}\label{eq:U_mToTheLeftOF0}
(1-\|q_{n_m}\alpha\|,1)\times\{0\}\subseteq U_{\alpha,m}.
\end{equation}
and
\begin{equation}\label{eq:U_mToTherightOF0}
(0,\|q_{n_m}\alpha\|)\times\{1\}\subseteq U_{\alpha,m}
\end{equation}
\end{enumerate}
\end{lemma}
\begin{proof}
\noindent {\tiny$\square$} \textit{Proof of (i)}: That $U_{\alpha,m}$ and $V_{\alpha,m}$ are $T_{\alpha,m}$-invariant follows from \cite[Lemma 4.1]{chaika}.\\
\noindent {\tiny$\square$} \textit{Proof of (ii)}: 
We utilize the recursive equations \eqref{eq:DefnU_m} and \eqref{eq:DefnV_m}. First we note that the  end-points of the components of the sets $T_{\alpha,r-1}^{i_r}J'_{\alpha,r}$, $r\in\{1,...,m\}$ and $i_r\in\{0,...,q_{n_r}-1\}$, belong to
$$\Delta_m=\{(k\alpha \mod 1,j)\,|\,k\in\{0,...,2\sum_{s=1}^mq_{n_s}-1\}\text{ and }j\in\{0,1\}\}.$$
Thus, \eqref{eq:ShapeOfUandV} holds.\\
To see that the set of  discontinuities of $\chi_{U_{\alpha,m}}$ and $\chi_{V_{\alpha,m}}$ each coincides with $\Delta_m$, we proceed by induction on $m\in\N$. When $m=1$, the result follows from \eqref{eq:DefnU_m} and \eqref{eq:DefnV_m}. Now let $m\in\N$ and suppose that $\Delta_m$ is the set of discontinuities of $\chi_{U_{\alpha,m}}$ and $\chi_{V_{\alpha,m}}$. Let $A=\bigcup_{i=0}^{q_{n_{(m+1)}}-1}T^i_{\alpha,m}J'_{\alpha,(m+1)}$. Set 
$\gamma=\chi_{U_{\alpha,m}}(1-\chi_A)$ and $\psi=(1-\chi_{U_{\alpha,m}})\chi_A$. Since for each $(x,j)\in\Delta_{m+1}\setminus\Delta_m$, $\chi_{U_{\alpha,m}}$ is continuous at $(x,j)$, we have that  for each $k\in B:=\{2\sum_{s=1}^mq_{n_s},...,2\sum_{s=1}^{m+1}q_{n_s}-1\}$ and $j\in\{0,1\}$, $\chi_{U_{\alpha,m}}$ is continuous at $(k\alpha,j)$. It follows that for each $k\in B$ and $j\in\{0,1\}$ and any open neighborhood $\mathcal O$ of $(k\alpha,j)$, one can find $y, z\in \mathcal O$ with $\chi_A(y)=0$, $\chi_A(z)=1$,  
$$\gamma(y)=\chi_{U_{\alpha,m}}(k\alpha,j)\text{ and }\psi(y)=0,$$
and 
$$\gamma(z)=0\text{ and }\psi(z)=1-\chi_{U_{\alpha,m}}(k\alpha,j)$$
Noting that 
$$\chi_{U_{\alpha,(m+1)}}=\gamma+\psi,$$
we see that every element of $\Delta_{m+1}\setminus\Delta_m$ is a discontinuity of $\chi_{U_{\alpha,(m+1)}}$.\\
Observe now that, since $\gamma=0$ on $A$ and $\psi=0$ on $(\T\times\Z_2)\setminus A$, every discontinuity of $\psi$ on the interior of  $A$  and every discontinuity of $\gamma$ on the interior of $(\T\times\Z_2)\setminus A$ is a discontinuity of  $\chi_{U_{\alpha,(m+1)}}$. Thus, $\Delta_m$ is a subset of the set of  discontinuities of $\chi_{U_{\alpha,(m+1)}}$ and, hence,  $\Delta_{m+1}$ is the set of discontinuities of  $\chi_{U_{\alpha,(m+1)}}$ (and, by a similar argument, of $\chi_{V_{\alpha,(m+1)}}$). We are done. \\
\noindent {\tiny$\square$} \textit{Proof of (iii)}: This can be easily checked by  induction on $m\in\N\cup\{0\}$. \\
\noindent {\tiny$\square$} \textit{Proof of (iv)}: First note that the sets $J'_{\alpha,1}$,...,$T_{\alpha,0}^{q_{n_1}-1}J'_{\alpha,1}$ are pairwise disjoint and $J'_{\alpha,1}=[0,\|q_{n_1}\alpha\|)\times\Z_2$. Thus, by \eqref{eq:DefnU_m},
\begin{equation}\label{eq:rightSideOFZeroU_1}
[0,\|q_{n_1}\alpha\|)\times\{1\}\subseteq U_{\alpha,1}.
\end{equation}
Note now that for each $r\in\{1,...,m\}$ and each $i_r\in\{0,...,q_{n_r}-1\}$,
$$T_{\alpha,{r-1}}^{i_r}J'_{\alpha,r}=[(2\sum_{s=0}^{r-1}q_{n_s}+i_r)\alpha,(2\sum_{s=0}^{r-1}q_{n_s}+q_{n_r}+i_r)\alpha)\times\Z_2.$$
Since  $2\sum_{s=1}^rq_{n_s}<q_{n_r+1}$, we have that if $(2\sum_{s=0}^{r-1}q_{n_s}+i_r)\neq 0$, then 
\begin{equation}\label{eq:AwayFromZero}
\|(2\sum_{s=0}^{r-1}q_{n_s}+i_r)\alpha\|,\|(2\sum_{s=0}^{r-1}q_{n_s}+q_{n_r}+i_r)\alpha\|\geq \|q_{n_r}\alpha\|.
\end{equation}
Noting that $(2\sum_{s=0}^{r-1}q_{n_s}+i_r)=0$ only when $r=1$ and $i_r=0$ and that 
$$q_{n_r}\alpha\mod 1>0$$
for each $r\in\{1,...,m\}$, 
we see that for every $r\in\{1,...,m\}$ and every $i_r\in\{0,...,q_{n_r}-1\}$ with $r\neq 1$ or $i_r\neq 0$,
$$\|q_{n_r}\alpha\|\leq (2\sum_{s=0}^{r-1}q_{n_s}+i_r)\alpha\mod 1<(2\sum_{s=0}^{r-1}q_{n_s}+q_{n_r}+i_r)\alpha\mod 1\leq1-\|q_{n_r}\alpha\|$$
Thus, since we also have 
$$0<q_{n_1}\alpha\mod 1\leq 1-\|q_{n_1}\alpha\|,$$
we obtain from \eqref{eq:DefnU_m} that
$$(1-\|q_{n_m}\alpha\|,1)\times\{0\}\subseteq U_{\alpha,0}\setminus \bigcup_{r=1}^m\bigcup_{i_r=0}^{q_{n_r}-1} T_{\alpha,(r-1)}^{i_r}J'_{\alpha,r}\subseteq U_{\alpha,m}$$
and by \eqref{eq:rightSideOFZeroU_1}, when $m>1$,
$$(0,\|q_{n_m}\alpha\|)\times\{1\}\subseteq U_{\alpha,1}\setminus\bigcup_{r=2}^m\bigcup_{i_r=0}^{q_{n_r}-1} T_{\alpha,(r-1)}^{i_r}J'_{\alpha,r}\subseteq U_{\alpha,m}$$
proving that  \eqref{eq:U_mToTheLeftOF0} and \eqref{eq:U_mToTherightOF0} hold. 
\end{proof}
\begin{remark}\label{rem:HierarchyOFConditions}
 Note that the condition that $2\sum_{s=1}^m\|q_{n_s}\alpha\|<1$ in (i) of Lemma \ref{lem:BasicFactsAboutUalphaTalpha} is weaker than the condition that for each $r\in\{1,...,m\}$, $2\sum_{s=1}^rq_{n_s}<q_{n_r+1}$ in (iv)  of Lemma \ref{lem:BasicFactsAboutUalphaTalpha}. Indeed, all we need to show is that  for any given $m\in\N$, whenever $2q_{n_1}<q_{n_1+1}$, one has $2\sum_{s=1}^m\|q_{n_s}\alpha\|<1$. To see this, note that since $(n_k)_{k\in\N}$ is an increasing sequence in $2\N$, we have that for each $s\in\N$, $2q_{n_s+1}<q_{n_{(s+1)}+1}$. Thus, 
\begin{equation}\label{eq:BoundOnSumOfClosestIntegers}
\sum_{s=k+1}^\infty\|q_{n_s}\alpha\|<\sum_{s=k+1}^\infty\frac{1}{q_{n_s+1}}<\frac{1}{q_{n_{(k+1)}+1}}\sum_{s=0}^\infty\frac{1}{2^s}=\frac{2}{q_{n_{(k+1)}+1}},
\end{equation}
for $k=0,1,...$. Since $2q_{n_1}<q_{n_1+1}$ implies that $a_{n_1+1}\geq 2$, we obtain that $q_{n_1+1}\geq 2q_{n_1}+q_{n_1-1}$. But $q_{n_1}\geq q_2\geq 2$ and $q_{n_1-1}\geq q_1\geq 1$ and, so,
$$2\sum_{s=1}^m\|q_{n_s}\alpha\|<\frac{4}{q_{n_1+1}}\leq \frac{4}{5}<1,$$
proving the claim.
\end{remark}
 Let $(b_n)_{n=1}^\infty$ be a sequence in $\N$. For any $\ell\in\N$ we define 
$$\cA_\ell(b_1,...,b_\ell)=\cA_\ell:=\{\alpha\in [0,1)\;:\; a^\alpha_i=b_i, \text{ for every } i\in\{1,...,\ell\}\}.
$$
Note that if $\alpha,\beta\in \cA_\ell(b_1,...,b_\ell)$ for some sequence $(b_n)_{n=1}^\infty$, then $p_n^\alpha=p_n^\beta$ and  $q_n^\alpha=q_n^\beta$ for $n=1,2,...,\ell$.\\
\begin{namedtheorem}{Proposition C}\hypertarget{prop:unbalancedB}{}
Let $(b_n)_{n\in\N}$ be a sequence in $\N$ and let $m\in\N$ be such that $m>1$. Form the sequence $(q_n)_{n\in\N}$ defined recursively by 
$$q_{n+1}=b_{n+1}q_n+q_{n-1},$$
$q_{-1}=0$, and $q_0=1$. Let $(n_k)_{k\in\N}$ be an increasing sequence in $2\N$ such that (a)  $b_{n_m+1}=3$ and (b)  $2\sum_{s=1}^{r}q_{n_s}<q_{n_{r}+1}$ for each $r\in\{1,...,m\}$. Then,  for any $\ell\in\N$ with $q_\ell>q_{n_m}^2$, any $\alpha\in \cA_\ell(b_1,...,b_\ell)$, any $n\geq\ell$, and any $(x,j)\in U_{\alpha,m}$ with
$$\min\{\|x+k\alpha\|\,|\,k\in\{0,...,q^\alpha_{n}-1\}\}\geq \frac{1}{16q_n^\alpha},$$
one has, 
\begin{equation}\label{eq:KeyEStimateOnA_ell}
\left| S_{q_n^\alpha}(T_{\alpha,m},h_1')(x,j)+q^\alpha_n\log(q_n^\alpha)-q^\alpha_n\Phi\right|<155q^\alpha_n,
\end{equation}
where $h_1$ is defined as in \eqref{eq:h} and $\Phi=\Phi_m(b_1,...,b_{n_m+1},n_1,...,n_m)\in\R$ satisfies 
\begin{equation}\label{eq:BoundsForGlobalPhi}
\frac{\log(q_{n_m})}{10}-18q_{n_{(m-1)}+1}\leq  \Phi
\leq 10\log(q_{n_m})+18q_{n_{(m-1)}+1}.
\end{equation}
\end{namedtheorem}
\subsection{Construction of $(a_k)_{k\in\N}$, $(n_k)_{k\in\N}$, and $(t_k)_{k\in\N}$}

To prove Proposition  \hyperlink{prop:PropertiesOFAlpha0}{B},
we will first construct inductively the sequences $(a_k)_{k\in\N}$, $(n_k)_{k\in\N}$, and $(t_k)_{k\in\N}$. Then we will show that the irrational number $\alpha_0$ with continued fraction expansion $[0;a_1,...]$ and denominators sequence $(q_n)_{n\in\N}$ satisfies conditions ($\alpha$.1)-($\alpha$.8) for some sequences $(y_k)_{k\in\N}$ and $(j_k)_{k\in\N}$ and  a constant $A>1$.

The following lemma provides us with sequences  $(a_k)_{k\in\N}$, $(n_k)_{k\in\N}$, and $(t_k)_{k\in\N}$ satisfying conditions ($\alpha$.4')-($\alpha$.7') and ($\beta$.1)-($\beta$.4) below.  Conditions ($\alpha$.4')-($\alpha$.7') have as immediate consequences conditions ($\alpha$.4)-($\alpha$.7) in  Proposition  \hyperlink{prop:PropertiesOFAlpha0}{B} and can be used to prove condition ($\alpha$.8). Conditions ($\beta$.1)-($\beta$.4) are needed to prove conditions ($\alpha$.1)-($\alpha$.3). 
\begin{lemma}\label{lem:ConstructionOfa_kn_kt_k}
 There exist $(a_k)_{k\in\N}$ in $\N$, a strictly increasing sequences  $(n_k)_{k\in\N}$ in $2\N$,  and  a strictly increasing sequence $(t_k)_{k\in\N}$ in $\N$ such that for any  $k\in\N$, 
\begin{enumerate}
    \item [($\alpha$.4')]  $\log(q_{t_k})/a_{(t_k+1)}<\frac{1}{k}$
    \item [($\alpha$.5')]  $a_{n_1+1}=a_{n_{(2k+1)}+1}=3$.
    \item [($\alpha$.6')]  (i) $\frac{1}{q_{n_{2k}}}\sum_{s=1}^{2k-1}q_{n_s}<\frac{1}{2k}$ and (ii) $\frac{1}{q_{n_{(2k+1)}}}\sum_{s=1}^{2k}q_{n_s}<\frac{1}{2k+1}$.
    \item [($\alpha$.7')]  $a_{n_{2k}+1}>k$.
    \item [($\beta$.1)]  $2\sum_{s=1}^{2k}q_{n_s}<q_{n_{2k}+1}$.
    \item [($\beta$.2)]  $\frac{1}{10}-\frac{18q_{n_{2k}+1}}{\log(q_{n_{(2k+1)}})}>\frac{1}{15}$. 
    \item [($\beta$.3)]  $q_{n_{(2k+1)}}^2<q_{t_k}$ and  $|\log(q_{t_k})-60\Phi_k|<\log(2)$, where 
    $$\Phi_k=\Phi_{2k+1}(a_1,...,a_{n_{(2k+1)}+1},n_1,...,n_{(2k+1)})$$
    is defined as in \eqref{eq:KeyEStimateOnA_ell}.
    \item [($\beta$.4)] If $k>1$, $t_{k-1}<n_{2k}$. 
\end{enumerate}
\end{lemma}
\begin{proof}
We will construct the sequences $(a_k)_{k\in\N}$, $(n_k)_{k\in\N}$, and $(t_k)_{k\in\N}$ inductively. At the $N$-th stage of the construction, we will pick $n_{2N}, n_{(2N+1)}$, $t_N$, and $a_{t_{(N-1)}+2},...,a_{t_N+1}$  so 
that $n_1,...,n_{(2N+1)}$, $t_1,...,t_N$, and $a_1,...,a_{t_N+1}$ satisfy 
conditions ($\alpha$.4')-($\alpha$.7') and ($\beta$.1)-($\beta$.4) for $k=1,...,N$. When $N=1$, we will also need to pick $n_1$, $a_{n_1+1}$, and let $t_{0}=-1$. 

\underline{Base Case}: For the first stage of the construction (i.e. $N=1$), we need to find the numbers $n_1,n_2, n_3$, $t_1$, and $a_1,...,a_{t_1+1}$. Note that since $N=1$, condition ($\beta$.4) holds trivially. \\
\noindent {\tiny$\square$} \textit{Conditions ($\alpha$.5'), ($\alpha$.7') and  ($\beta$.1)}: For this,  let  $n_1=2$, $n_2=4$, and let $\alpha_1$ be an irrational number with continued fraction expansion $[0;a^{(1)}_1,a^{(1)}_2,...]$ such that for $k\not\in \{3,5\}$, $a^{(1)}_k=1$, $a_3^{(1)}=3$, and $a_5^{(1)}>2$ is large enough to ensure that
\begin{equation}\label{eq.a.9'k=1}
2(q^{\alpha_1}_{n_1}+q^{\alpha_1}_{n_2})\leq a^{(1)}_5q^{\alpha_1}_{n_2}=a^{(1)}_5q^{\alpha_1}_4<q^{\alpha_1}_{n_2+1}=q^{\alpha_1}_5.
\end{equation}
\noindent {\tiny$\square$} \textit{Conditions ($\alpha$.6') item (ii) and  ($\beta$.2)}: 
Pick now $n_3\in 2\N$ large enough to ensure that $5<n_3$,
\begin{equation}\label{eq.a.6'k=3}
\frac{\sum_{s=1}^2q^{\alpha_1}_{n_s}}{q^{\alpha_1}_{n_{3}}}<\frac{1}{3}
\end{equation}
and
\begin{equation}\label{eq:LowerBoundForPhi_1}
\frac{1}{10}-\frac{18q^{\alpha_1}_{n_{2}+1}}{\log(q^{\alpha_1}_{n_{3}})}>\frac{1}{15}.
\end{equation}
\noindent {\tiny$\square$} \textit{Condition ($\alpha$.5')}: Let $\beta_1$ be the irrational number with continued fraction expansion $[0;b^{(1)}_1,b^{(1)}_2,...]$ where 
$b^{(1)}_k=a^{(1)}_k$ if $k\neq n_3+1$ and $b_{n_3+1}^{(1)}=3$ (so, in particular $b_k^{(1)}=1$ for $k\geq n_3+1$).\\
\noindent {\tiny$\square$} \textit{Condition ($\beta$.3)}: Let $(n'_k)_{k\in\N}$ be an increasing sequence in $2\N$ with $n'_k=n_k$ for $k=1,2,3$. Since $b^{(1)}_k=a^{(1)}_k$ for $k\leq n_3$, 
$$2q^{\beta_1}_{n_1}<3q^{\beta_1}_{n_1}+q^{\beta_1}_{n_1-1}=q_{n_1+1}^{\beta_1}$$
and, by \eqref{eq.a.9'k=1} and our choice of $b_{n_3+1}^{(1)}=3$,
$$2\sum_{s=1}^3q_{n_s}^{\beta_1}<2q_{n_3}^{\beta_1}+q_{n_2+1}^{\beta_1}<3q_{n_3}^{\beta_1}<q_{n_3+1}.$$
Thus, the hypothesis of Proposition \hyperlink{prop:unbalancedB}{C} holds with $m=3$  and $(b_n)_{n\in\N}=(b^{(1)}_k)_{k\in\N}$.\\

Let $\Phi_1=\Phi_3(b_1^{(1)},...,b_{n_3+1}^{(1)},n'_1,n'_2,n'_3)$ be the constant guaranteed to exist by Proposition \hyperlink{prop:unbalancedB}{C}. Combining \eqref{eq:BoundsForGlobalPhi} and \eqref{eq:LowerBoundForPhi_1}, we obtain 
\begin{equation}\label{eq:LowerBoundForPhi_1PartII}
\frac{\log(q_{n_3}^{\beta_1})}{15}<\Phi_1.
\end{equation}
We claim that there exits $t_1\in \N$ with $t_1>(n_3+1)$, $q^{\beta_1}_{t_1}>(q^{\beta_1}_{n_3})^2$ and
\begin{equation}\label{eq:BaseCaseNearLog}
q^{\beta_1}_{t_1}\in [\frac{1}{2}e^{60\Phi_1},2e^{60\Phi_1}].
\end{equation}
Indeed, since $4<q^{\beta_1}_{n_3}$, 
$$q^{\beta_1}_{n_3+1}=3q^{\beta_1}_{n_3}+q^{\beta_1}_{n_3-1}<4q^{\beta_1}_{n_3}<(q^{\beta_1}_{n_3})^2.$$
Thus, by \eqref{eq:LowerBoundForPhi_1PartII},
$$\log(2q_{n_3+1}^{\beta_1})<\log(2(q^{\beta_1}_{n_3})^2)<2\log((q^{\beta_1}_{n_3})^2)<4\log(q^{\beta_1}_{n_3})<60\Phi_1.$$
It follows that $q_{n_3+1}^{\beta_1},(q_{n_3}^{\beta_1})^2<\frac{1}{2}e^{60\Phi_1}$. Noting that for $t>n_3+1$, $q^{\beta_1}_{t-1}<q^{\beta_1}_t<2q^{\beta_1}_{t-1}$, we see that there exists $t_1>n_3+1$ such that \eqref{eq:BaseCaseNearLog} holds,
or, equivalently, 
\begin{equation}\label{eq:CloseToPhi1}
|\log(q^{\beta_1}_{t_1})-60\Phi_1|<\log(2).
\end{equation}
\noindent {\tiny$\square$} \textit{Condition ($\alpha$.4')}:
Set $a_k=b_k^{(1)}$ for $k\leq t_1$ and  let $a_{t_1+1}$ be the smallest natural number such that 
$\log(q^{\beta_1}_{t_1})/a_{t_1+1}<1$.\\
 \noindent {\tiny$\square$} \textit{Checking conditions ($\alpha$.4')-($\alpha$.7') and ($\beta$.1)-($\beta$.4) for $(a_k)_{k=1}^{t_1+1}$}: Recall that for each $n\in\{1,...,t_k+1\}$, $q_{n+1}=a_{n+1}q_n+q_{n-1}$, where $q_0=1$ and $q_{-1}=0$. Thus, $q_n^{\beta_1}=q_n^{\alpha_1}$ for $n\in\{1,...,n_3\}$ and $q_n=q_n^{\beta_1}$ for $n\in\{1,...,t_1\}$.  It follows that  $\log(q_{t_1})/a_{t_1+1}=\log(q^{\beta_1}_{t_1})/a_{t_1+1}<1$, which implies that ($\alpha$.4') holds. We also have $a_{n_1+1}=a_{n_3+1}=b_{n_3+1}^{(1)}=3$, so ($\alpha$.5') holds. To see that ($\alpha$.6') holds, first note that 
 $$2q_2=2q_{n_1}<q_3+q_2<q_{4}$$
 and, hence, (i) in ($\alpha$.6') holds. By \eqref{eq.a.6'k=3}, (ii) in ($\alpha$.6') also holds. Since $a_5>2$, ($\alpha$.7') holds. By \eqref{eq.a.9'k=1}, ($\beta$.1) holds. By \eqref{eq:LowerBoundForPhi_1}, ($\beta$.2) holds. Condition ($\beta$.3) follows from our choice of $t_1$, our definition of $\Phi_1$, and  \eqref{eq:CloseToPhi1}. Since $N=1$,  ($\beta$.4) holds trivially. \\

\underline{Inductive Step}: Fix now $N\in\N$ and suppose that we have picked  $n_1<\cdots<n_{2N+1}$, $t_1<\cdots<t_N$, and $a_j$ for $j=1,...,t_N+1$ in such a way that conditions ($\alpha$.4')-($\alpha$.7') and ($\beta$.1)-($\beta$.4) are satisfied for $k\leq \N$. We want to find $n_{2N+2}<n_{2N+3}$, $t_{N+1}$, and $a_j$, $j=t_{N}+2,...,t_{N+1}+1$, such that conditions $(\alpha$.4')-($\alpha$.7') and  $(\beta$.1)-($\beta$.4) hold for $k=N+1$.\\
\noindent {\tiny$\square$} \textit{Conditions ($\alpha$.6') item (i) and ($\beta$.4)}:
For this, let $\delta_{N+1}$ be the irrational number with continued fraction expansion $[0;d_1^{(N+1)},...]$ where $d_k^{(N+1)}=a_k$ for $k\leq t_{N}+1$  and $d_k^{(N+1)}=1$ for $k>t_{N}+1$. Pick $n_{(2N+2)}>t_{N}+1$ large enough to ensure that 
$$\frac{\sum_{s=1}^{2N+1}q^{\delta_{N+1}}_{n_s}}{q^{\delta_{N+1}}_{n_{(2N+2)}}}<\frac{1}{2N+2}.$$
\noindent {\tiny$\square$} \textit{Conditions ($\alpha$.7') and ($\beta$.1)}:
Now let $\alpha_{N+1}$ be the irrational number with continued fraction expansion $[0;a_1^{(N+1)},...]$ where $a_k^{(N+1)}=d_k^{(N+1)}$ for $k\leq n_{(2N+2)}$, $a_k^{(N+1)}=1$ for $k>n_{(2N+2)}+1$, and $a^{(N+1)}_{n_{(2N+2)}+1}>2N+2$ is large enough to ensure that 
\begin{equation}\label{eq:BoundOnLargestEven}
2\sum_{s=1}^{2N+2}q^{\alpha_{N+1}}_{n_s}<q^{\alpha_{N+1}}_{n_{(2N+2)}+1}
\end{equation}
\noindent {\tiny$\square$} \textit{Conditions ($\alpha$.6') item (ii) and ($\beta$.2)}:
 Pick now $n_{(2N+3)}\in 2\N$, $n_{(2N+3)}>n_{(2N+2)}+1$, large enough to ensure that
$$\frac{\sum_{s=1}^{2N+2}q^{\alpha_{N+1}}_{n_s}}{q^{\alpha_{N+1}}_{n_{(2N+3)}}}<\frac{1}{2N+3}$$
 and
\begin{equation}\label{eq:LowerBoundForPhi_m}
\frac{1}{10}-\frac{18q^{\alpha_{N+1}}_{n_{(2N+2)}+1}}{\log(q^{\alpha_{N+1}}_{n_{(2N+3)}})}>\frac{1}{15}.
\end{equation}
\noindent {\tiny$\square$} \textit{Condition ($\alpha$.5')}:
Let $\beta_{N+1}$ be the irrational number with continued fraction expansion $[0;b^{(N+1)}_1,b^{(N+1)}_2,...]$ where 
$b^{(N+1)}_k=a^{(N+1)}_k$ if $k\neq n_{2N+3}+1$ and $b_{n_{(2N+3)}+1}^{(N+1)}=3$.\\
\noindent {\tiny$\square$} \textit{Condition ($\beta$.3)}:
Let $(n'_k)_{k\in\N}$ be an increasing sequence in $2\N$ with $n'_k=n_k$ for $k=1,...,n_{(2N+3)}$. Since $b^{(N+1)}_k=a^{(N+1)}_k$ for $k\leq n_{2N+3}$, condition ($\beta$.1) and \eqref{eq:BoundOnLargestEven} imply that for 
 every even $r\in\{1,...,2N+3\}$, $2\sum_{s=1}^rq_{n_s}^{\beta_{N+1}}<q_{n_r+1}^{\beta_{N+1}}$. Thus, by condition ($\alpha$.5') and $b^{(N+1)}_{n_{(2N+3)}+1}=3$, for every odd $r\in\{1,...,2N+3\}$, $$2\sum_{s=1}^rq_{n_s}^{\beta_{N+1}}=2\sum_{s=0}^rq_{n_s}^{\beta_{N+1}}=2\sum_{s=0}^{r-1}q_{n_s}^{\beta_{N+1}}+2q_{n_r}^{\beta_{N+1}}\leq q_{n_{(r-1)}+1}^{\beta_{N+1}}+2q_{n_r}^{\beta_{N+1}}<q_{n_r+1}^{\beta_{N+1}},$$
 where $q_{n_0}^{\beta_{N+1}}=q_{n_0+1}^{\beta_{N+1}}=0$.
It follows that the hypothesis of Proposition \hyperlink{prop:unbalancedB}{C} holds with $m=2N+3$  and $(b_n)_{n\in\N}=(b^{(N+1)}_k)_{k\in\N}$. Let $\Phi_{N+1}=\Phi_{(2N+3)}(b^{(N+1)}_1,...,b^{(N+1)}_{n_{(2N+3)}+1}, n'_1,...,n'_{(2N+3)})$ be the constant guaranteed to exist by Proposition \hyperlink{prop:unbalancedB}{C}. Arguing as before, we can find a $t_{N+1}>n_{(2N+3)}+1$ such that $q_{t_{(N+1)}}^{\beta_{N+1}}>(q_{n_{(2N+3)}}^{\beta_{N+1}})^2$ and 
$$|\log(q^{\beta_{N+1}}_{t_{(N+1)}})-60\Phi_{N+1}|<\log(2).$$
\noindent {\tiny$\square$} \textit{Condition ($\alpha$.4')}:
Set $a_k=b_k^{(N+1)}$ for $k\leq t_{N+1}$ and
let $a_{t_{(N+1)}+1}$ be the smallest natural number such that 
$\log(q^{\beta_{N+1}}_{t_{(N+1)}})/a_{t_{(N+1)}+1}<\frac{1}{N+1}$.\\
We now complete the induction by checking that $n_1,...,n_{(2N+3)}$, $t_1$,...,$t_{(N+1)}$, and $a_1$,... ,$a_{t_{(N+1)}+1}$ satisfy conditions ($\alpha$.4')-($\alpha$.7') and ($\beta$.1)-($\beta$.4). The proof of Lemma \ref{lem:ConstructionOfa_kn_kt_k} is thus complete.
\end{proof}
\subsection{Proof of Proposition B}
We can now formulate and prove a more precise statement of Proposition B.
\begin{namedtheorem}{Proposition B$^\prime$}\hypertarget{prop:PropertiesOFAlpha0'}{}
Let the sequences $(n_k)_{k\in\N}$, $(a_k)_{k\in\N}$ and $(t_k)_{k\in\N}$ be given by Lemma \ref{lem:ConstructionOfa_kn_kt_k}. Let $\alpha_0=[0;a_1,...]$ and  $A=\frac{60}{59}$.
There exist a sequence   $(y_k)_{k\in\N}$ in $\mathbb T$,  a sequence $(j_k)_{k\in\N}$ in $\{0,1\}$, such that  conditions ($\alpha$.1)-($\alpha$.8) in Proposition  \hyperlink{prop:PropertiesOFAlpha0}{B} are satisfied.
\end{namedtheorem}
\begin{proof}
We divide the proof of Proposition  \hyperlink{prop:PropertiesOFAlpha0'}{B$^\prime$} into various steps.

\noindent {\tiny$\square$} \textit{Proof of conditions ($\alpha$.4)-($\alpha$.7)}: That conditions 
 ($\alpha$.4)-($\alpha$.7) hold, follows immediately from  ($\alpha$.4')-($\alpha$.7').\\
\noindent {\tiny$\square$} \textit{Proof of condition ($\alpha$.8)}: To see that condition ($\alpha$.8) holds, we will prove that
\begin{equation}\label{eq:Alpha.8Even}
\lim_{k\rightarrow\infty}(q_{n_{2k}}\sum_{s=2k+1}^\infty\|q_{n_s}\alpha_0\|)=0
\end{equation}
and
\begin{equation}\label{eq:Alpha.8Odd}
\lim_{k\rightarrow\infty}(q_{n_{(2k+1)}}\sum_{s=2k+2}^\infty\|q_{n_s}\alpha_0\|)=0.
\end{equation}
By \eqref{eq:BoundOnSumOfClosestIntegers} and ($\alpha$.7), we obtain
\begin{multline*}
\lim_{k\rightarrow\infty}(q_{n_{2k}}\sum_{s=2k+1}^\infty\|q_{n_s}\alpha_0\|)\leq \lim_{k\rightarrow\infty}\frac{2q_{n_{2k}}}{q_{n_{(2k+1)}+1}}\\
\leq \lim_{k\rightarrow\infty} \frac{2q_{n_{2k}}}{q_{n_{2k}+1}}\leq \lim_{k\rightarrow\infty} \frac{2q_{n_{2k}}}{a_{n_{2k}+1}q_{n_{2k}}}=\lim_{k\rightarrow\infty}\frac{2}{a_{n_{2k}+1}}=0,
\end{multline*}
and
\begin{multline}\label{eq:OddToEvenTimeSize}
\lim_{k\rightarrow\infty}(q_{n_{(2k+1)}}\sum_{s=2k+2}^\infty\|q_{n_s}\alpha_0\|)\leq \lim_{k\rightarrow\infty}\frac{2q_{n_{(2k+1)}}}{q_{n_{(2k+2)}+1}}\\
\leq \lim_{k\rightarrow\infty} \frac{2q_{n_{(2k+1)}}}{a_{n_{(2k+2)}+1}q_{n_{(2k+2)}}}
\leq\lim_{k\rightarrow\infty}\frac{2}{a_{n_{(2k+2)}+1}}=0,
\end{multline}
proving that \eqref{eq:Alpha.8Even} and \eqref{eq:Alpha.8Odd} hold.\\

\noindent {\tiny$\square$} \textit{Preamble to the construction of $(y_k)_{k\in\N}$ and $(j_k)_{k\in\N}$}: We now construct the sequences $(y_k)_{k\in\N}$ and $(j_k)_{k\in\N}$. To do this we will first show the following three facts which hold for $m\in\N$ large enough: (a) The set where $T_{\alpha_0}$ coincides with $T_{\alpha_0,2m+1}$ has measure close to 1, (b) The set of points in $U_{\alpha_0,2m+1}$ which are $\|q_{t_m}\alpha_0\|$-away from the discontinuities of $T_{\alpha_0,2m+1}$ is large, and (c) For a majority of $z\in\T\times\Z_2$, $(T_{\alpha_0}^jz)_{j=-1}^{q_{t_m}-1}$ is $\frac{1}{16q_{t_m}}$-away from discontinuities of $T_{\alpha_0}$.\\

\noindent {\tiny$\square$} \textit{For large $m$, the set where $T_{\alpha_0}$ coincides with $T_{\alpha_0,2m+1}$ is large}. Indeed, note that for large enough $m\in\N$,
\begin{multline*}
    \lambda(\{(x,j)\in\T\times\Z_2\,|\,\forall k\in\{0,...,2q_{t_m}-1\},\, T_{\alpha_0,2m+1}^k(x,j)=T_{\alpha_0}^k(x,j)\})\\
    \geq 1-\lambda(\bigcup_{k=0}^{2q_{t_m}-1}\left[[2\sum_{s=1}^{2m+1}q_{n_s}\alpha_0,2\sum_{s=1}^\infty q_{n_s}\alpha_0]-k\alpha_0\right]\times\Z_2)\geq 1-2q_{t_m}\sum_{s=2m+2}^\infty\|q_{n_s}\alpha\|. 
\end{multline*}
Arguing as in \eqref{eq:OddToEvenTimeSize} and noting that, by ($\beta$.4), $t_{m}<n_{(2m+2)}$ for each $m\in\N$, we obtain 
\begin{multline*}
    \lim_{k\rightarrow\infty}(q_{t_{k}}\sum_{s=2k+2}^\infty\|q_{n_s}\alpha_0\|)\leq \lim_{k\rightarrow\infty}\frac{2q_{t_{k}}}{q_{n_{(2k+2)}+1}}\\
\leq \lim_{k\rightarrow\infty} \frac{2q_{n_{(2k+2)}}}{a_{n_{(2k+2)}+1}q_{n_{(2k+2)}}}
=\lim_{k\rightarrow\infty}\frac{2}{a_{n_{(2k+2)}+1}}=0.
\end{multline*}
So,
\begin{equation}\label{eq:AssymptoticT_mApproxToT}
\lim_{m\rightarrow\infty}\lambda(\{(x,j)\in\T\times\Z_2\,|\,\forall k\in\{0,...,2q_{t_m}-1\},\, T_{\alpha_0,2m+1}^k(x,j)=T_{\alpha_0}^k(x,j)\})=1.
\end{equation}
\noindent {\tiny$\square$} \textit{For large $m$, the set of points in $U_{\alpha_0,2m+1}$ which are $\|q_{t_m}\alpha_0\|$-away from the discontinuities of $T_{\alpha_0,2m+1}$ is large}:
Let $m\in\N$. Observe that, by Lemma \ref{lem:BasicFactsAboutUalphaTalpha} items (ii) and (iii), the set $U_{\alpha_0,2m+1}$ has at most $N_m=2\sum_{s=1}^{2m}q_{n_s}+2q_{n_{(2m+1)}}$ components and that for every $x\in\mathbb T$, there exists exactly one  $j\in\{0,1\}$ with $(x,j)\in U_{\alpha_0,2m+1}$.
Furthermore, the discontinuities of $\chi_{U_{\alpha_0,2m+1}}$ can only occur at points of the form $(k\alpha_0,j)$ where $k\in\{0,...,N_m-1\}$ and $j\in\Z_2$. Let $c^{(m)}_1,...,c^{(m)}_{N_m}\in [0,1)$, $0=c^{(m)}_1<c^{(m)}_2<\cdots<c^{(m)}_{N_m}$ be an ordered enumeration of the elements of $\Gamma_m=\{k\alpha_0\mod 1\,|\,k\in\{0,...,N_m-1\}\}$. It follows that for some $j_{m,1},...,j_{m,N_m}\in \Z_2$,
\begin{equation}\label{eq:a.s.DescriptionOFU_mAsIntervals}
\lambda(U_{\alpha_0,2m+1}\triangle\left(\bigcup_{s=1}^{N_m}(c^{(m)}_{s},c^{(m)}_{s+1})\times\{j_{m,s}\}\right))=0,
\end{equation}
where $c^{(m)}_{N_m+1}=1$. Observe that by ($\beta$.1) and ($\alpha$.5'), $N_m<3q_{n_{(2m+1)}}<q_{n_{(2m+1)}+1}$ and, by  ($\beta$.3), $q_{n_{(2m+1)}}^2<q_{t_m}$. Thus, since $12< q_6\leq q_{n_{(2m+1)}}$,
\begin{multline*}
2\|q_{t_m}\alpha_0\|<\frac{2}{q_{t_m+1}}<\frac{2}{q_{n_{(2m+1)}}^2}<\frac{1}{6q_{n_{(2m+1)}}}\\
<\frac{1}{q_{n_{(2m+1)}+1}+q_{n_{(2m+1)}}}<\|q_{n_{(2m+1)}}\alpha_0\|\leq |c^{(m)}_{s+1}-c^{(m)}_s|,
\end{multline*}
for each $s\in\{1,...,N_m\}$. It follows that for each $s\in\{1,...,N_m\}$,
$$(c^{(m)}_s+\|q_{t_m}\alpha_0\|,c^{(m)}_{s+1}-\|q_{t_m}\alpha_0\|)$$
is well-defined and, hence, 
\begin{multline}\label{eq:LongComponentOfU_m}
\liminf_{m\rightarrow\infty}\lambda(\bigcup_{s=1}^{N_m}\left((c^{(m)}_{s}+\|q_{t_m}\alpha_0\|,c^{(m)}_{s+1}-\|q_{t_m}\alpha_0\|)\times\{j_{m,s}\}\right))\\
\geq \lim_{m\rightarrow\infty}(\frac{1}{2}-\frac{2N_m}{q_{t_m+1}})
\geq \lim_{m\rightarrow\infty}(\frac{1}{2}-\frac{6q_{n_{(2m+1)}}}{q_{n_{(2m+1)}}^2})=\frac{1}{2}.
\end{multline}
\noindent {\tiny$\square$} \textit{For large $m$ and  a majority of $z\in\T\times\Z_2$, $(T_{\alpha_0}^jz)_{j=-1}^{q_{t_m}-1}$ is $\frac{1}{16q_{t_m}}$-away form discontinuities of $T_{\alpha_0}$}:
For each $m\in\N$, let
\begin{multline}\label{eq:DefnOmega_m}
\Omega_m=\{(x,j)\in\T\times\Z_2\,|\\
\exists k\in\{0,...,q_{t_m}\},\, \|x+(k-1)\alpha_0-x_1\|,\|x+(k-1)\alpha_{0}-x_0|\|< \frac{1}{16q_{t_m}}\},
\end{multline}
 where $x_0$ and $x_1$ are as defined in \eqref{eq:discontinuitiesOfT}. Clearly, 
\begin{equation}\label{eq:AwayFromDiscontinuitiesInB'}
\limsup_{k\rightarrow\infty}\lambda(\Omega_k)\leq \limsup_{m\rightarrow\infty}\frac{q_{t_k}+1}{4q_{t_k}}= \frac{1}{4}.
\end{equation}
\noindent {\tiny$\square$} \textit{Choosing $(y_k)_{k\in\N}$ and $(j_k)_{k\in\N}$}: 
By \eqref{eq:AssymptoticT_mApproxToT}, \eqref{eq:a.s.DescriptionOFU_mAsIntervals}, \eqref{eq:LongComponentOfU_m}, and \eqref{eq:AwayFromDiscontinuitiesInB'}, there exists $m_0\in\N$ such that for any $m\in\N$ with $m>m_0$, we can find $(y_m,j_m)\in \T\times\Z_2$ and $\ell=\ell_m\in\{1,...,N_m\}$ such that  
\begin{equation}\label{eq:EnsureReturnToSameComponent}
(y_m,j_m)\in (c^{(m)}_\ell+\|q_{t_m}\alpha_0\|,c^{(m)}_{\ell+1}-\|q_{t_m}\alpha_0\|)\times \{j_{m,\ell}\}\subseteq U_{\alpha_0,2m+1},
\end{equation}
 for every $k\in\{0,...,2q_{t_m}-1\}$, $T^k_{\alpha_0}(y_m,j_m)=T^k_{\alpha_0,2m+1}(y_m,j_m)$, and $(y_m,j_m)\in \Omega_m^c$. For every $m\leq m_0$, we set 
$(y_m,j_m)=(0,0)$. \\
 \noindent {\tiny$\square$} \textit{Proof of condition ($\alpha$.1)}: By conditions ($\alpha$.5') 
 and ($\beta$.1) one has that for any $k\in\N$, $a_{n_{(2k+1)}+1}=3$ and for any $r\in\{1,...,2k+1\}$, $2\sum_{s=1}^rq_{n_s}<q_{n_r+1}$. Thus, when one replaces $(b_n)_{n\in\N}$ with $(a_n)_{n\in\N}$ in the statement of Proposition \hyperlink{prop:unbalancedB}{C}, both conditions (a) and (b) are satisfied for any number of the form $m=2k+1$, $k\in\N$. Furthermore, by 
  ($\beta$.3),  $q_{t_k}>q_{n_{(2k+1)}}^2$ for every $k\in\N$ and, by our choice of $((y_k,j_k))_{k\in\N}$, we have that for $k>m_0$, $(y_k,j_k)\in U_{\alpha_0,2k+1}$ and
 $$\min\{\|y_k+s\alpha_0\|\,|\,s\in\{0,...,q_{t_k}-1\}\}\geq \frac{1}{16q_{t_k}}.$$
Thus, letting $A=\frac{60}{59}$, we obtain from Proposition  \hyperlink{prop:unbalancedB}{C} and condition ($\beta$.3) that for $k>m_0$,
\begin{multline*}
\left| S_{q_{t_k}}(T_{{\alpha_0},2k+1},h_1')(y_k,j_k)+A^{-1}q_{t_k}\log(q_{t_k})\right|\\
=\left| S_{q_{t_k}}(T_{{\alpha_0},2k+1},h_1')(y_k,j_k)+q_{t_k}(1-\frac{1}{60})\log(q_{t_k})\right|\\
\leq \left| S_{q_{t_k}}(T_{{\alpha_0},2k+1},h_1')(y_k,j_k)+q_{t_k}\log(q_{t_k})-q_{t_k}\Phi_k\right|+q_{t_k}|\frac{\log(q_{t_k})}{60}-\Phi_k|\\
<q_{t_k}(155+\frac{\log(2)}{60}). 
\end{multline*}
So, there exits $B>1$ such that for $k>m_0$, 
$$\left|S_{q_{t_k}}(T_{{\alpha_0},2k+1},h_A')(y_k,j_k)+q_{t_k}\log(q_{t_k})\right|<Bq_{t_k}.$$
Condition (${\alpha}$.1) now follows from $T^s_{\alpha_0}(y_k,j_k)=T^s_{\alpha_0,2k+1}(y_k,j_k)$ for every $s\in\{0,...,2q_{t_k}-1\}$.\\
 \noindent {\tiny$\square$} \textit{Proof of condition ($\alpha$.2)}: Substituting  $x_0=-\alpha_0$ and $x_1=-\alpha_0+|J_{\alpha_0}(n_k)|$ in \eqref{eq:DefnOmega_m}, we see that  ($\alpha$.2) holds.\\
 \noindent {\tiny$\square$} \textit{Proof of condition ($\alpha$.3)}:
Note that  for any $k>m_0$, there exists a $j'_k\in\Z_2$ with  
$$(y_k+q_{t_k}\alpha_0,j'_k)=T^{q_{t_k}}_{\alpha_0}(y_k,j_k)=T^{q_{t_k}}_{\alpha_0,2k+1}(y_k,j_k).$$
Since, by (i) in Lemma \ref{lem:BasicFactsAboutUalphaTalpha} and Remark \ref{rem:HierarchyOFConditions}, $U_{\alpha_0,2k+1}$ is $T_{\alpha_0,2k+1}$-invariant and $(y_k,j_k)\in U_{\alpha_0,2k+1}$, we obtain that 
$$(y_k+q_{t_k}\alpha_0,j'_k)\in  U_{\alpha_0,2k+1}.$$ 
So, by \eqref{eq:EnsureReturnToSameComponent}, $(y_k+q_{t_k}\alpha_0,j'_k)$ and $(y_k,j_k)$ lie in the interior of the same component of $U_{\alpha_0,2k+1}$ and, hence, $j'_k=j_k$, proving that ($\alpha$.3) holds. Proposition B is thus proved. \end{proof}

\section{Proof of Proposition C}\label{sec:proofOFpropC}
In this section we prove Proposition \hyperlink{prop:unbalancedB}{C}. Let $\alpha\in (0,1)$ be an irrational number, let $(n_k)_{k\in\N}$ be an increasing sequence in $2\N$,  and let  $T_{\alpha,s}$ be defined as in \eqref{eq:SthSkew}. Recall that  $U_{\alpha,0}=\mathbb T\times\{0\}$, $V_{\alpha,0}=\mathbb T\times\{1\}$, $T_{\alpha,0}(x,j)=(x+\alpha,j)$, and for any $m\in\N$, $J'_{\alpha,m}$, $U_{\alpha,m}$, and $V_{\alpha,m}$ are defined as in \eqref{eq:DefnJ'_m}, \eqref{eq:DefnU_m}, and \eqref{eq:DefnV_m}.\\

For $m\in\N$ define $\phi_{\alpha,m}:\T\rightarrow\T$ by 
\begin{equation}\label{eq:DefnFunctionPhi}
\phi_{\alpha,m}(x)=\frac{\chi_{[1/2,1)}(x)-\chi_{(0,1/2)}(x)}{\|x\|}\chi_{U_{\alpha,m}}(x,1),
\end{equation}
The following proposition reduces the study of the Birkhoff sums of 
$h'_1$ (see \eqref{eq:h})  over $T_{\alpha,m}$ to those of $\phi_{\alpha,m}$ over $R_{\alpha}$.
\begin{proposition}\label{prop:h_1=phi}
  For any irrational $\alpha\in (0,1)$, any increasing sequence $(n_k)_{k\in\N}$ in $2\N$, any $m\in\N$, and any $(x,j)\in U_{\alpha,m}$ with $x\neq 0$, we have
  \begin{equation}\label{eq:h_1=phi}
h_1'(x,j)=\phi_{\alpha,m}(x).
\end{equation}
  Furthermore, if $U_{\alpha,m}$ is $T_{\alpha,m}$-invariant,  $n\in\N$ and $h_1'$ is defined at each of  $(x+k\alpha\mod 1,j)$, $k=0,...,n-1$ and $j\in\{0,1\}$, then
\begin{equation}\label{eq:EquivalentErgSums}
S_{n}(T_{\alpha,m},h_1')(x,j)=\sum_{j=0}^{n-1} h_1'(T_{\alpha,m}^j(x,j))=\sum_{j=0}^{n-1}\phi_{\alpha,m}(R_{\alpha}^jx)=S_{n}(R_{\alpha},\phi_{\alpha,m})(x).
\end{equation}  
\end{proposition}
\begin{proof}
 By item (iii) in Lemma \ref{lem:BasicFactsAboutUalphaTalpha}, for every $y\in\T$ there exists a unique $i\in\Z_2$ with $(y,i)\in U_{\alpha,m}$. Thus, for any  $(x,j)\in U_{\alpha,m}$, with $x\not\in \{0,1/2\}$, we have that when $j=0$,
 \begin{multline*}
 h'_1(x,j)=\chi_{\{1\}}(j)\frac{\chi_{[1/2,1)}(x)-\chi_{(0,1/2)}(x)}{\|x\|}=0\\
 =\frac{\chi_{[1/2,1)}(x)-\chi_{(0,1/2)}(x)}{\|x\|}\chi_{U_{\alpha,m}}(x,1)=\phi_{\alpha,m}(x)
 \end{multline*}
 and when $j=1$,
  \begin{multline*}
 h'_1(x,j)=\chi_{\{1\}}(j)\frac{\chi_{[1/2,1)}(x)-\chi_{(0,1/2)}(x)}{\|x\|}=\frac{\chi_{[1/2,1)}(x)-\chi_{(0,1/2)}(x)}{\|x\|}\\
 =\frac{\chi_{[1/2,1)}(x)-\chi_{(0,1/2)}(x)}{\|x\|}\chi_{U_{\alpha,m}}(x,1)=\phi_{\alpha,m}(x)
 \end{multline*}
 That \eqref{eq:EquivalentErgSums} holds follows from \eqref{eq:h_1=phi} and the fact that $U_{\alpha,m}$ is $T_{\alpha,m}$-invariant. 
\end{proof}
The following two lemmas will be needed for the proof of Proposition \hyperlink{prop:unbalancedB}{C}. 
\begin{lemma}\label{lemProp27ForSingleAlpha}
Let $\alpha\in (0,1)$ be an irrational number, let  $(n_k)_{k\in\N}$ be an increasing sequence in $2\N$, and let $m,M>1$. Suppose that (a) $3\leq a_{n_m+1}\leq M$ and (b) $2\sum_{s=1}^{r}q_{n_s}<q_{n_{r}+1}$ for each $r\in\{1,...,m\}$. Then, 
for any $n\in\N$ with $n>n_m+1$ and any  $(x,j)\in U_{\alpha,m}$ with 
$$\min\{\|x+k\alpha\|\,|\,k\in\{0,...,q_{n}-1\}\}\geq \frac{1}{16q_n},$$
one has, 
\begin{equation}\label{eqQuadraticBoundOFErgodicSum}
\left| S_{q_n}(T_{\alpha,m},h_1')(x,j)+q_n\log(q_n)-q_n\Phi_{\alpha,m}\right|\leq 43q_n+64q_{n_m}^2,
\end{equation}
where
\begin{equation}\label{eq:DefnConstantPhiAlpha,m}
    \Phi_{\alpha,m}=\int_\T \left(\phi_{\alpha,m}(x)+\frac{\chi_{[0,\|q_{n_{(m-1)}}\alpha\|)}(x)}{\|x\|}\right)\text{d}\lambda'(x)-\log(\|q_{n_{(m-1)}}\alpha\|).
\end{equation}
 satisfies 
\begin{equation}\label{eqBoundsForPhialpham}
\frac{\log(q_{n_m})}{2(M+2)}-18q_{n_{(m-1)}+1}\leq  \Phi_{\alpha,m}\\
\leq 4\frac{\log(q_{n_m})}{M}+18q_{n_{(m-1)}+1}.
\end{equation}
\end{lemma}

Let $(b_n)_{n=1}^\infty$ be a sequence in $\N$. Recall that for any $\ell\in\N$, 
$$\cA_\ell(b_1,...,b_\ell)=\cA_\ell:=\{\alpha\in [0,1)\;:\; a^\alpha_i=b_i, \text{ for every } i\in\{1,...,\ell\}\}.
$$
\begin{lemma}\label{lem:BoundForDifferentAlphas}
Let $(b_n)_{n\in\N}$ be a sequence in $\N$ and let $m\in\N$ be such that $m>1$. Form the sequence $(q_n)_{n\in\N}$ defined recursively by 
$$q_{n+1}=b_{n+1}q_n+q_{n-1},$$
where $q_{-1}=0$, and $q_0=1$. Let $(n_k)_{k\in\N}$ be an increasing sequence in $2\N$ such that (a) For some $M\in\N$, $3\leq b_{n_m+1}\leq M$ and (b) $2\sum_{s=1}^{r}q_{n_s}<q_{n_{r}+1}$ for each $r\in\{1,...,m\}$. For any $\ell>n_m$ with $q_\ell>q_{n_m}^2$ and any $\alpha,\beta\in \cA_\ell(b_1,...,b_\ell)$,
\begin{equation}\label{eq:DiscrepancyOfPhis}
|\Phi_{\alpha,m}-\Phi_{\beta,m}|<12(M+1).
\end{equation}
\end{lemma}
\begin{remark}\label{rem:q_ell>q_nmImplies_ellLarge}
We remark that when $b_{n_m+1}=3$ and $m>1$, the condition $q_\ell>q_{n_m}^2$ in Proposition \hyperlink{prop:unbalancedB}{C} and Lemma \ref{lem:BoundForDifferentAlphas} implies that $\ell>n_m+1$. Indeed, since $n_m$ is even and $m>1$, we must have that $n_m\geq 4$ and, hence, $q_{n_m}\geq 5$. So, since $b_{n_m+1}=3$, 
$$q_{n_m+1}=3q_{n_m}+q_{n_m-1}<4q_{n_m}<q_{n_m}^2<q_\ell.$$
Noting that $(q_n)_{n\in\N}$ is an increasing sequence, we conclude that $\ell>n_{m}+1$. 
\end{remark}

We will prove the above lemmas in the subsections below, let us first show how they imply Proposition \hyperlink{prop:unbalancedB}{C}.

\begin{proof}[{\bf Proof of Proposition \hyperlink{prop:unbalancedB}{C}}]
Let 
$$\ell_0=\min\{\ell\in\N\,|\,q_\ell>q_{n_m}^2\},$$
let $\alpha\in \cA_{\ell_0}(b_1,...,b_{\ell_0})$, and set $\Phi=\Phi_{\alpha,m}$, where $\Phi_{\alpha,m}$ is defined in  \eqref{eq:DefnConstantPhiAlpha,m}. It follows from \eqref{eqBoundsForPhialpham} with $M=3$ that $\Phi$ satisfies \eqref{eq:BoundsForGlobalPhi}.\\
Now let $\ell\in\N$ be such that $q_\ell>q_{n_m}^2$ and pick $\beta\in \cA_\ell(b_1,...,b_\ell)$. Since $\ell\geq \ell_0$, we have that $\beta \in \cA_{\ell_0}(b_1,...,b_{\ell_0})$. 
It now follows from Lemma \ref{lemProp27ForSingleAlpha} and Remark \ref{rem:q_ell>q_nmImplies_ellLarge} that for any $n\geq \ell_0>n_m+1$ and any $(x,j)\in U_{\beta,m}$ with
$$\min\{\|x+k\beta\|\,|\,k\in\{0,...,q_{n}^\beta-1\}\}\geq \frac{1}{16q^\beta_n},$$
one has, 
\begin{multline*}
\left| S_{q^\beta_n}(T_{\beta,m},h_1')(x,j)+q^\beta_n\log(q_n^\beta)-q^\beta_n\Phi\right|=\left| S_{q^\beta_n}(T_{\beta,m},h_1')(x,j)+q^\beta_n\log(q_n^\beta)-q^\beta_n\Phi_{\alpha,m}\right|\\
\leq \left| S_{q^\beta_n}(T_{\beta,m},h_1')(x,j)+q^\beta_n\log(q_n^\beta)-q^\beta_n\Phi_{\beta,m}\right|+q^\beta_n\left|\Phi_{\alpha,m}-\Phi_{\beta,m}\right|<107q^\beta_n+48q^\beta_n, 
\end{multline*}
where in the last inequality we used \eqref{eq:DiscrepancyOfPhis}.
We are done. 
\end{proof}
\subsection{Proofs of Lemmas  \ref{lemProp27ForSingleAlpha} and \ref{lem:BoundForDifferentAlphas}}
Denote by $\pi$ the canonical projection from $\T\times\Z_2$ to $\T$ (so, $\pi(x,j)=x$). We will find it convenient to  express $\phi_{\alpha,m}$ with the help of the following sets:
\begin{equation}\label{eq:defnAm}
A_m=\pi\big(U_{\alpha,(m-1)}\cap ([\|q_{n_{(m-1)}}\alpha\|,1/2)\times\{1\})\big),
\end{equation}
\begin{equation}\label{eq:defnBm}
B_m=\pi\big(U_{\alpha,(m-1)}\cap ([1/2,1-\|q_{n_{(m-1)}}\alpha\|)\times\{1\})\big),
\end{equation}
\begin{equation}\label{eq:defnCm}
C_m=\pi\big(U_{\alpha,m}\setminus U_{\alpha,(m-1)}\cap ([\|q_{n_{(m-1)}}\alpha\|,1/2)\times\{1\}) \big),
\end{equation}
\begin{equation}\label{eq:defnDm}
D_m=\pi\big(U_{\alpha,(m-1)}\setminus U_{\alpha,m}\cap ([1/2,1-\|q_{n_{(m-1)}}\alpha\|)\times\{1\})\big),
\end{equation}
\begin{equation}\label{eq:defnEm}
E_m=\pi\big((U_{\alpha,m}\triangle U_{\alpha,(m-1)})\cap (\mathbb T\times\{1\})\big).
\end{equation}
and 
\begin{equation}\label{eq:defnFm}
F_m=[0,\|q_{n_{(m-1)}}\alpha\|).
\end{equation}
\begin{sublemma}\label{lem:ExpresionForh'}
Let $m\in\N$ be such that $m>1$ and let $\alpha\in(0,1)$ be  an irrational number. Suppose that  $2\sum_{s=1}^{r}q_{n_s}<q_{n_{r}+1}$ for each $r\in\{1,...,m\}$. 
Then 
\begin{equation}\label{eq:DecompositionOFphi'}
    \phi_{\alpha,m}(x)=-\frac{\chi_{F_m}}{\|x\|}+
    \frac{\chi_{E_m}}{\|x\|}-\frac{\chi_{A_m}}{\|x\|}- 
    \frac{2\chi_{C_m}}{\|x\|}
    +\frac{\chi_{B_m}}{\|x\|}-\frac{2\chi_{D_m}}{\|x\|},
\end{equation} 
at every $x\in\T\setminus\{0\}$.
\end{sublemma}
While the decomposition of $\phi_{\alpha,m}$ in \eqref{eq:DecompositionOFphi'} does not correspond to any partition of $[0,1)$, the ergodic sum of each of its summands can either be estimated with enough  precision or shown to have a "small" contribution to the total ergodic sum of $\phi_{\alpha,m}$. As a matter of fact, the main contributors to the ergodic sum of $\phi_{\alpha,m}$ are the terms $-\frac{\chi_{F_m}}{\|x\|}$ and $\frac{\chi_{E_m}}{\|x\|}$. The following diagram illustrates why this is the case. 
 \begin{figure}[H]
 \centering
\resizebox{!}{6.5cm} {\includegraphics{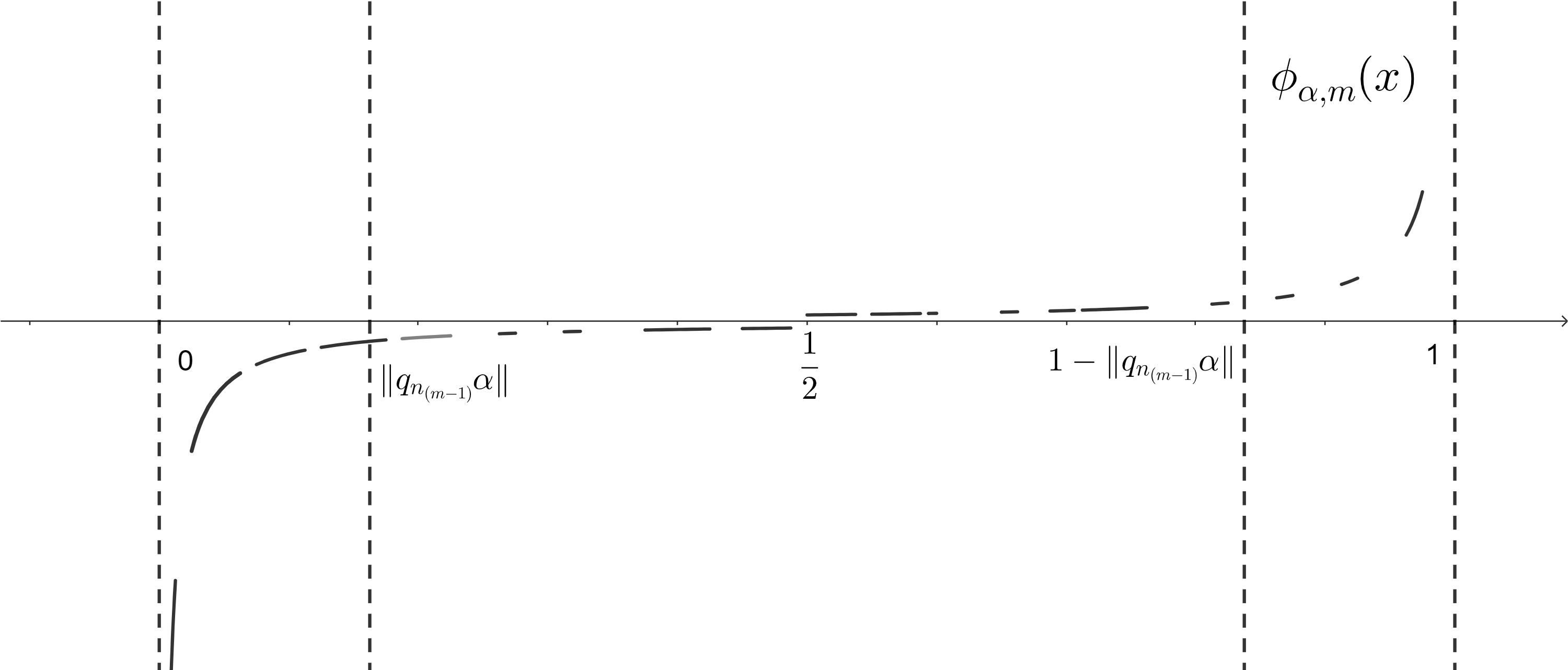}}
\caption{\small The diagram above provides a visual representation of the graph of $\phi_{\alpha,m}$ on the interval $[0,1)$. Note that the portion of the graph of $\phi_{\alpha,m}$ on $[0,\|q_{n_{(m-1)}}\alpha\|)$ is unbounded and the restriction of $\phi_{\alpha,m}$ to this interval coincides with that of $-\frac{\chi_{F_m}}{\|x\|}+\frac{\chi_{E_m}}{\|x\|}$. The restriction of $\phi_{\alpha,m}$ to $[1-\|q_{n_{(m-1)}}\alpha\|,1)$ is, so to say, "nearly" unbounded and it coincides with the restriction of $\frac{\chi_{E_m}}{\|x\|}$ to  $[1-\|q_{n_{(m-1)}}\alpha\|,1)$. Finally, the restriction  of $\phi_{\alpha,m}$ to  $[\|q_{n_{(m-1)}}\alpha\|,1-\|q_{n_{(m-1)}}\alpha\|)$ can be bounded by a relatively small constant  and coincides with the restriction of $\frac{\chi_{E_m}}{\|x\|}-\frac{\chi_{A_m}}{\|x\|}- 
    \frac{2\chi_{C_m}}{\|x\|}
    +\frac{\chi_{B_m}}{\|x\|}-\frac{2\chi_{D_m}}{\|x\|}$ to the same interval. 
}
\end{figure}

We will also need the following bounds on various quantities involving the sets $A_m$--$E_m$.
\begin{sublemma}\label{lem:BoundForTotalVarOfBoundedSets}
Let $m,M\in\N$ be such that $m>1$ and $M\geq 3$ and let $\alpha\in(0,1)$ be an irrational number. Suppose that (a) $a_{n_m+1}= M$ and (b) $2\sum_{s=1}^{r}q_{n_s}<q_{n_{r}+1}$ for each $r\in\{1,...,m\}$. Then
\begin{enumerate}
\item[V1.]\vspace{-1em}
$$\text{Var}(\chi_{A_m}\frac{1}{\|x\|}),\text{Var}(\chi_{B_m}\frac{1}{\|x\|}), \text{Var}(\chi_{C_m}\frac{1}{\|x\|}), \text{Var}(\chi_{D_m}\frac{1}{\|x\|})<8q_{n_m}q_{n_{(m-1)}+1};$$
\item[V2.] \vspace{-1em}
$$\text{Var}(\chi_{E_m}\frac{1}{\|x\|})< 8(q_{n_m+1}+q_{n_m}+q_{n_m}\log(q_{n_m}));$$
\item [V3.]\vspace{-1em}
$$\frac{\log(q_{n_m})}{2(M+2)}\leq \int_{E_m}\frac{1}{\|x\|}\text{d}\lambda'(x)\leq 2+4\frac{1+\log(q_{n_m})}{M}.$$\end{enumerate}
\end{sublemma}

We will now use Sublemmas \ref{lem:ExpresionForh'} and \ref{lem:BoundForTotalVarOfBoundedSets}  to prove Lemmas  \ref{lemProp27ForSingleAlpha} and \ref{lem:BoundForDifferentAlphas}. We will prove Sublemma \ref{lem:ExpresionForh'} in Subsection \ref{sec:exfor} and Sublemma \ref{lem:BoundForTotalVarOfBoundedSets} in Subsection \ref{sec:lse}. The proof of Lemma  \ref{lemProp27ForSingleAlpha} is carried out in the next two subsubsections, and the proof of Lemma \ref{lem:BoundForDifferentAlphas} will be given in Subsection \ref{subsec713}.
\subsubsection{Proof of \eqref{eqQuadraticBoundOFErgodicSum} in Lemma \ref{lemProp27ForSingleAlpha}}
\begin{proof}
Let $(x,j)\in U_{\alpha,m}$ be such that $\|x+k\alpha\|\geq \frac{1}{16q_n}$ for each $k\in\{0,...,q_n-1\}$. By \eqref{eq:EquivalentErgSums} and Sublemma \ref{lem:ExpresionForh'}, 
\begin{multline*}
S_{q_n}(T_{\alpha,m},h_1')(x,j)=S_{q_n}(R_\alpha,\phi_{\alpha,m})(x)\\
=-S_{q_n}(R_\alpha,\frac{\chi_{F_m}}{\|\cdot\|})(x)+
    S_{q_n}(R_\alpha,\frac{\chi_{E_m}}{\|\cdot\|})(x)-S_{q_n}(R_\alpha,\frac{\chi_{A_m}}{\|\cdot\|})(x)\\
    -
    2S_{q_n}(R_\alpha,\frac{\chi_{C_m}}{\|\cdot\|})(x)
    +S_{q_n}(R_\alpha,\frac{\chi_{B_m}}{\|\cdot\|})(x)-2S_{q_n}(R_\alpha,\frac{\chi_{D_m}}{\|\cdot\|})(x),
\end{multline*}
Note that since $\|x+k\alpha\|\geq \frac{1}{16q_n}$  for each $k\in\{0,...,q_n-1\}$,
$$S_{q_n}(R_\alpha,\frac{\chi_{F_m}}{\|\cdot\|})(x)=S_{q_n}(R_\alpha,\frac{\chi_{[\frac{1}{16q_n},\|q_{n_{(m-1)}}\alpha\|)}}{\|\cdot\|})(x).$$
By Denjoy-Koksma inequality,
\begin{multline}\label{eq:DenjoyKoksma1ForS'}
\left|S_{q_n}\left(R_{\alpha},\chi_{[\frac{1}{16q_n},\|q_{n_{(m-1)}}\alpha\|)}\frac{-1}{\|x\|}\right)(x)+q_n(\log(q_n)+\log(16)+\log(\|q_{n_{(m-1)}}\alpha\|))\right|\\
=\left|S_{q_n}\left(R_{\alpha},\chi_{[\frac{1}{16q_n},\|q_{n_{(m-1)}}\alpha\|)}\frac{-1}{\|x\|}\right)(x)-q_n\int_{1/(16q_n)}^{\|q_{n_{(m-1)}}\alpha\|}\frac{-1}{\|y\|}\text{d}\lambda'(y)\right|\\
\leq \text{Var}(\chi_{[\frac{1}{16q_n},\|q_{n_{(m-1)}}\alpha\|]}\frac{1}{\|x\|})\leq 32q_n,
\end{multline}
Moreover, for $F\in\{A_m,B_m,C_m,D_m\}$,  V1. in Sublemma \ref{lem:BoundForTotalVarOfBoundedSets} implies 
\begin{equation}\label{eq:DenjoyKoksma2ForS'}
    \left|S_{q_n}\left(R_{\alpha},\frac{\chi_{F}}{\|x\|}\right)(x)- q_n\int_{F}\frac{1}{\|y\|}\text{d}\lambda'(y)\right|\leq \text{Var}(\chi_{F}\frac{1}{\|x\|})\leq  8q_{n_m}q_{n_{(m-1)}+1}\leq 8q_{n_m}^2
\end{equation}
and V2. in Sublemma \ref{lem:BoundForTotalVarOfBoundedSets} implies that
\begin{multline}\label{eq:DenjoyKoksma3ForS'}
 \left|S_{q_n}\left(R_{\alpha},\frac{\chi_{E_m}}{\|x\|}\right)(x)- q_n\int_{E_m}\frac{1}{\|y\|}\text{d}\lambda'(y)\right|\\
 \leq \text{Var}(\chi_{E_m}\frac{1}{\|x\|})\leq  8(q_{n_m+1}+q_{n_m}+q_{n_m}\log(q_{n_m}))\leq 8q_n+16q_{n_m}^2.
\end{multline}
Combining \eqref{eq:DefnConstantPhiAlpha,m}, \eqref{eq:DenjoyKoksma1ForS'}, \eqref{eq:DenjoyKoksma2ForS'}, and \eqref{eq:DenjoyKoksma3ForS'},  we obtain
\begin{multline*}
    \left|S_{q_n}(T_{\alpha,m},h_1')(x,j)+q_n\log(q_n)-q_n\Phi_{\alpha,m}\right|=\left|S_{q_n}(R_{\alpha},\phi_{\alpha,m})(x)+q_n\log(q_n)-q_n\Phi_{\alpha,m}\right|\\
\leq \left|S_{q_n}(R_{\alpha},\frac{-\chi_{[\frac{1}{16q_n},\|q_{n_{(m-     1)}}\alpha\|)}}{\|\cdot\|})(x)+q_n(\log(q_n)+\log(16)+\log(\|q_{n_{(m-1)}}\alpha\|))\right|+q_n\log(16)\\
+\left|S_{q_n}(R_{\alpha},\phi_{\alpha,m}+\frac{\chi_{[\frac{1}{16q_n},\|q_{n_{(m-1)}}\alpha\|)}}{\|\cdot\|})(x)-q_n\int_\T \left(\phi_{\alpha,m}(x)+\frac{\chi_{[0,\|q_{n_{(m-1)}}\alpha\|)}(x)}{\|x\|}\right)\text{d}\lambda'(x)\right|\\
    \leq  32q_n+\log(16)q_n+48q_{n_m}^2+8q_n+16q_{n_m}^2< 43q_n+64q_{n_m}^2.
\end{multline*}
We are done. 
\end{proof}
\subsubsection{Proof of (\ref{eqBoundsForPhialpham}) in Lemma \ref{lemProp27ForSingleAlpha}}
\begin{proof} Using Sublemma \ref{lem:ExpresionForh'} and the definition of $\Phi_{\alpha,m}$,

\begin{multline*}
   \Phi_{\alpha,m}=\int_\T \left(\phi_{\alpha,m}(x)+\frac{\chi_{[0,\|q_{n_{(m-1)}}\alpha\|)}(x)}{\|x\|}\right)\text{d}\lambda'(x)-\log(\|q_{n_{(m-1)}}\alpha\|)\\    
    =\int_{E_m}\frac{1}{\|x\|}\text{d}\lambda'
    +\int_{A_m}\frac{-1}{\|x\|}\text{d}\lambda'+ 
    \int_{C_m}\frac{-2}{\|x\|}\text{d}\lambda'
    +\int_{B_m}\frac{1}{\|x\|}\text{d}\lambda'+ \int_{D_m}\frac{-2}{\|x\|}\text{d}\lambda'-\log(\|q_{n_{(m-1)}}\alpha\|).
\end{multline*}
 Observe that by definition (see \eqref{eq:defnAm}--\eqref{eq:defnDm}) it follows that for $F\in \{A_m,B_m,C_m,D_m\}$ we have $F\cap [0,\|q_{n_{(m-1)}}\alpha\|)=\emptyset$ and $F\cap (1-\|q_{n_{(m-1)}}\alpha\|,1]=\emptyset$. Therefore 
 $$
 \Big|\int_{F}\frac{1}{\|x\|}\text{d}\lambda'\Big|\leq \frac{1}{\|q_{n_{(m-1)}}\alpha\|}\leq 2q_{n_{(m-1)}+1}.
 $$
 
Thus, by V3 in Sublemma \ref{lem:BoundForTotalVarOfBoundedSets},
\begin{multline*}
    \frac{\log(q_{n_m})}{2(M+2)}-\log(\|q_{n_{(m-1)}}\alpha\|)-12q_{n_{(m-1)}+1}
    \leq  \Phi_{\alpha,m}\\
    \leq 2+4\frac{1+\log(q_{n_m})}{M}-\log(\|q_{n_{(m-1)}}\alpha\|)+12q_{n_{(m-1)}+1}.
\end{multline*}
Since 
$$\log(q_{n_{(m-1)}+1})\leq-\log(\|q_{n_{(m-1)}}\alpha\|)
\leq \log(2q_{n_{(m-1)}+1}),$$
we obtain (\ref{eqBoundsForPhialpham}) and finish the proof of Lemma \ref{lemProp27ForSingleAlpha}.\\

\end{proof}
\subsubsection{Proof of Lemma \ref{lem:BoundForDifferentAlphas}} \label{subsec713}
\begin{proof}
Let $\ell>n_m$ be such that $q_\ell>q_{n_m}^2$ and pick $\alpha, \beta\in \cA_\ell(b_1,...,b_\ell)$. Let $\Phi_{\alpha,m}$, $\Phi_{\beta,m}$ be as defined in \eqref{eq:DefnConstantPhiAlpha,m}. Observe that, since $m>1$, 
$$\|q_{n_{(m-1)}}\alpha\|,\|q_{n_{(m-1)}}\beta\|
<\frac{1}{q_{n_{(m-1)}+1}}\leq \frac{1}{q_3}\leq \frac{1}{3}.$$
Thus, 
\begin{multline*}
\int_\T \left(\frac{\chi_{[0,\|q_{n_{(m-1)}}\alpha\|)}(x)}{\|x\|}-\frac{\chi_{[0,\|q_{n_{(m-1)}}\beta\|)}(x)}{\|x\|}\right)\text{d}\lambda'(x)-\log(\|q_{n_{(m-1)}}\alpha\|)+\log(\|q_{n_{(m-1)}}\beta\|)\\
=\log(\|q_{n_{(m-1)}}\alpha\|)-\log(\|q_{n_{(m-1)}}\beta\|)-\log(\|q_{n_{(m-1)}}\alpha\|)+\log(\|q_{n_{(m-1)}}\beta\|)=0.
\end{multline*}
Therefore, in order to prove \eqref{eq:DiscrepancyOfPhis}, it is enough to show that 
\begin{equation}\label{eq:12(M+2)DiscrepancyBound}
|\int_\T (\phi_{\alpha,m}(x)-\phi_{\beta,m}(x))\text{d}\lambda'|<12(M+1).
\end{equation}
Using \eqref{eq:DefnFunctionPhi} we obtain
\begin{multline*}
\int_\T (\phi_{\alpha,m}(x)-\phi_{\beta,m}(x))\text{d}\lambda'\\
=2\int_{\T\times \{1\}}\frac{\chi_{[1/2,1)}(x)-\chi_{[0,1/2)}(x)}{\|x\|}(\chi_{U_{\alpha,m}}-\chi_{U_{\beta,m}})(x,j)\text{d}\lambda(x,j)\\
=2\int_{\T\times \Z_2}\chi_{\{1\}}(j)\frac{\chi_{[1/2,1)}(x)-\chi_{[0,1/2)}(x)}{\|x\|}(\chi_{U_{\alpha,m}}-\chi_{U_{\beta,m}})(x,j)\text{d}\lambda(x,j).
\end{multline*}
Thus, in order to prove \eqref{eq:12(M+2)DiscrepancyBound} it is enough to show
\begin{equation}\label{eq:maa}
|\int_{\T\times \Z_2}\chi_{\{1\}}(j)\frac{\chi_{[1/2,1)}(x)-\chi_{[0,1/2)}(x)}{\|x\|}(\chi_{U_{\alpha,m}}-\chi_{U_{\beta,m}})(x,j)\text{d}\lambda(x,j)|
<6(M+1),
\end{equation}
To prove \eqref{eq:maa}, we assume that, without loss of generality, $\|q_{n_m}\alpha\|<\|q_{n_m}\beta\|$. Observe that by condition (b) in  Lemma \ref{lem:BoundForDifferentAlphas} and item (iv) in Lemma \ref{lem:BasicFactsAboutUalphaTalpha},
\begin{equation}\label{eq:EqualityNear0}
\chi_{U_{\alpha,m}}=\chi_{U_{\beta,m}}
\end{equation}
on $\left((0,\|q_{n_m}\alpha\|)\cup(1-\|q_{n_m}\alpha\|,1)\right)\times\Z_2$.
Hence, 
\begin{multline*}
    |\int \chi_{\{1\}}(j)\frac{\chi_{[1/2,1)}(x)-\chi_{[0,1/2)}(x)}{\|x\|}(\chi_{U_{\alpha,m}}-\chi_{U_{\beta,m}})(x,j)\text{d}\lambda(x,j)|\\
    \leq \int\chi_{[\|q_{n_m}\alpha\|,1-\|q_{n_m}\alpha\|]}(x)\chi_{\{1\}}(j)\frac{1}{\|x\|}|\chi_{U_{\alpha,m}}-\chi_{U_{\beta,m}}|(x,j)\text{d}\lambda(x,j)\\
    \leq \int 2q_{n_m+1}\chi_{\{1\}}(j)|\chi_{U_{\alpha,m}}-\chi_{U_{\beta,m}}|(x,j)\text{d}\lambda(x,j)\leq 2q_{n_m+1}\lambda(U_{\alpha,m}\triangle U_{\beta,m}).
\end{multline*}
By Proposition  
 \ref{prop:DistributionOfIrrationalsWithSimilarContinuedFraction},
 for each $k\in\{1,...,2\sum_{s=1}^mq_{n_s}-1\}$ there exists a  $c_k\in\{0,...,q_\ell-1\}$ such that 
 $$k\alpha,k\beta\mod 1\in (\frac{c_k}{q_\ell},\frac{c_k+1}{q_\ell}).$$
Thus, by item (ii) in Proposition \ref{lem:BasicFactsAboutUalphaTalpha}, for any $c\in\{1,...,q_\ell-1\}$, the number of discontinuities of $\chi_{U_{\alpha,m}}$ and $\chi_{U_{\beta,m}}$ on the interval $(0,\frac{c}{q_\ell}
)\times\{1\}$ is the same. Similarly, the number of discontinuities of $\chi_{U_{\alpha,m}}$ and $\chi_{U_{\beta,m}}$ on $(0,\frac{c}{q_\ell})\times\{0\}$ also coincide. Thus, when $c\not\in\{c_k\,|\,k\in\{1,...,2\sum_{s=1}^{m}q_{n_s}-1\}\}$, both $\chi_{U_{\alpha,m}}$ and $\chi_{U_{\beta,m}}$ are continuous on $\left(\frac{c}{q_\ell},\frac{c+1}{q_\ell}\right)\times\Z_2$ and, by 
 \eqref{eq:EqualityNear0},
$$\chi_{U_{\alpha,m}}=\chi_{U_{\beta,m}}$$
on $\left(\frac{c}{q_\ell},\frac{c+1}{q_\ell}\right)\times\Z_2$.
It follows that 
$$\lambda(U_{\alpha,m}\triangle U_{\beta,m})\leq \frac{2\sum_{s=1}^mq_{n_s}}{q_\ell}\leq \frac{3q_{n_m}}{q_\ell}.$$
So,
\begin{multline*}
    |\int\chi_{\{1\}}(j)\frac{\chi_{[1/2,1)}(x)-\chi_{[0,1/2)}(x)}{\|x\|}(\chi_{U_{\alpha,m}}-\chi_{U_{\beta,m}})(x,j)\text{d}\lambda(x,j)|\\
    \leq \frac{6q_{n_m+1}q_{n_m}}{q_\ell}\leq \frac{6(M+1)q_{n_m}^2}{q_\ell}< 6(M+1)
\end{multline*}
We are done. 
\end{proof}

\subsection{Proof of Sublemma \ref{lem:ExpresionForh'}}\label{sec:exfor}
\begin{proof}
Let $\psi,\tilde\psi:\mathbb T\rightarrow\mathbb T$ be defined by 
$$\psi(x)=(-\chi_{[0,1/2)}(x)+\chi_{[1/2,1)}(x))\chi_{U_{\alpha,m}}(x,1).$$
and 
$$\tilde \psi(x)=-\chi_{F_m}(x)+
    \chi_{E_m}(x)-\chi_{A_m}(x)- 2\chi_{C_m}(x)
    +\chi_{B_m}(x)-2\chi_{D_m}(x).$$
Note that for each $x\in\T\setminus\{0\}$, we have
$$\phi_{\alpha,m}(x)=\frac{\psi(x)}{\|x\|}.$$
Thus, in order to prove \eqref{eq:DecompositionOFphi'}, it suffices to show that $\psi=\tilde\psi$. 
We will do this by checking that $\psi=\tilde\psi$ on each of the intervals $[0,\|q_{n_{(m-1)}}\alpha\|)$, $[\|q_{n_{(m-1)}}\alpha\|,1/2)$, $[1/2,1-\|q_{n_{(m-1)}}\alpha\|)$, and $[1-\|q_{n_{(m-1)}}\alpha\|,1)$ (note that since $m>1$, $\|q_{n_{(m-1)}}\alpha\|<\frac{1}{2}$). Our main tool will be the identity
\begin{equation*}
\chi_{U_{\alpha,(m-1)}\setminus U_{\alpha,m}}+\chi_{U_{\alpha,m}\setminus U_{\alpha,(m-1)}}=|\chi_{U_{\alpha,(m-1)}\setminus U_{\alpha,m}}-\chi_{U_{\alpha,m}\setminus U_{\alpha,(m-1)}}|=\chi_{U_{\alpha,m}\triangle U_{\alpha,(m-1)}}.
\end{equation*}
\noindent {\tiny$\square$} \textit{The interval $[0,\|q_{n_{(m-1)}}\alpha\|)$}: Let $x\in[0,\|q_{n_{(m-1)}}\alpha\|)$. By item (iv) in Lemma \eqref{lem:BasicFactsAboutUalphaTalpha}, we have
$$1=\chi_{F_m}(x)=\chi_{U_{\alpha,(m-1)}}(x,1)\text{ and }0=\chi_{U_{\alpha,m}\setminus U_{\alpha,(m-1)}}(x,1).$$
Thus,
\begin{multline*}
\tilde\psi(x)=-\chi_{F_m}(x)+\chi_{E_m}(x)
=-\chi_{U_{\alpha,(m-1)}}(x,1)+\chi_{U_{\alpha,m}\triangle U_{\alpha,(m-1)}}(x,1)\\
=-\chi_{U_{\alpha,(m-1)}\setminus U_{\alpha,m}}(x,1)-\chi_{U_{\alpha,(m-1)}\cap U_{\alpha,m}}(x,1)+|\chi_{U_{\alpha,(m-1)}\setminus U_{\alpha,m}}-\chi_{U_{\alpha,m}\setminus U_{\alpha,(m-1)}}|(x,1)\\
=-\chi_{U_{\alpha,(m-1)}\cap U_{\alpha,m}}(x,1)+\chi_{U_{\alpha,m}\setminus U_{\alpha,(m-1)}}(x,1)\\
=-\chi_{U_{\alpha,(m-1)}\cap U_{\alpha,m}}(x,1)-\chi_{U_{\alpha,m}\setminus U_{\alpha,(m-1)}}(x,1)=-\chi_{U_{\alpha,m}}(x,1)=\psi(x).
\end{multline*}
\noindent {\tiny$\square$} \textit{The interval $[\|q_{n_{(m-1)}}\alpha\|,1/2)$}: Let  $x\in [\|q_{n_{(m-1)}}\alpha\|,1/2)$, we have
\begin{multline*}
    \tilde\psi(x)=\chi_{E_m}(x)-\chi_{A_m}(x)-2\chi_{C_m}(x)\\
    =|\chi_{U_{\alpha,(m-1)}\setminus U_{\alpha,m}}-\chi_{U_{\alpha,m}\setminus U_{\alpha,(m-1)}}|(x,1)-\chi_{U_{\alpha,(m-1)}}(x,1)-2\chi_{U_{\alpha,m}\setminus U_{\alpha,(m-1)}}(x,1)\\
    =-\chi_{U_{\alpha,(m-1)}}(x,1)+\chi_{U_{\alpha,(m-1)}\setminus U_{\alpha,m}}(x,1) 
    -\chi_{U_{\alpha,m}\setminus U_{\alpha,(m-1)}}(x,1)\\
    =-\chi_{U_{\alpha,m}\cap U_{\alpha,(m-1)}}(x,1)-\chi_{U_{\alpha,m}\setminus U_{\alpha,(m-1)}}(x,1)
    =-\chi_{U_{\alpha,m}}(x,1)=\psi(x).
\end{multline*}
\noindent {\tiny$\square$} \textit{The interval $[1/2,1-\|q_{n_{(m-1)}}\alpha\|)$}: Let  $x\in[1/2,1-\|q_{n_{(m-1)}}\alpha\|)$, we have
\begin{multline*}
    \tilde\psi(x)=\chi_{E_m}(x)+\chi_{B_m}(x)-2\chi_{D_m}(x)\\
    =|\chi_{U_{\alpha,(m-1)}\setminus U_{\alpha,m}}-\chi_{U_{\alpha,m}\setminus U_{\alpha,(m-1)}}|(x,1)+\chi_{U_{\alpha,(m-1)}}(x,1)-2\chi_{U_{\alpha,(m-1)}\setminus U_{\alpha,m}}(x,1)\\
    =\chi_{U_{\alpha,(m-1)}}(x,1)-\chi_{U_{\alpha,(m-1)}\setminus U_{\alpha,m}}(x,1)+\chi_{U_{\alpha,m}\setminus U_{\alpha,(m-1)}}(x,1)\\
    =\chi_{U_{\alpha,(m-1)}\cap U_{\alpha,m}}(x,1)+\chi_{U_{\alpha,m}\setminus U_{\alpha,(m-1)}}(x,1)
    =\chi_{U_{\alpha,m}}(x,1)=\psi(x).
\end{multline*}
\noindent {\tiny$\square$} \textit{The interval $[1-\|q_{n_{(m-1)}}\alpha\|,1)$}: Let  $x\in[1-\|q_{n_{(m-1)}}\alpha\|,1)$. By item (iv) in Lemma \ref{lem:BasicFactsAboutUalphaTalpha},
$$0=\chi_{U_{\alpha,(m-1)}}(x,1).$$
Thus,
\begin{multline*}
\tilde\psi(x)=\chi_{E_m}(x)=\chi_{U_{\alpha,(m-1)}\setminus U_{\alpha,m}}(x,1)+\chi_{U_{\alpha,m}\setminus U_{\alpha,(m-1)}}(x,1)\\
=\chi_{U_{\alpha,m}\setminus U_{\alpha,(m-1)}}(x,1)
=\chi_{U_{\alpha,m}\setminus U_{\alpha,(m-1)}}(x,1)+\chi_{U_{\alpha,m}\cap U_{\alpha,(m-1)}}(x,1)\\
=\chi_{U_{\alpha,m}}(x,1)=\psi(x).
\end{multline*}
We are done. 
\end{proof}
\subsection{Proof of Sublemma \ref{lem:BoundForTotalVarOfBoundedSets}}\label{sec:lse}
We will prove conditions V1-V3 separetely. 
\subsubsection{Proof of condition V1.}
\begin{proof}
Let $F$ be any of $A_m$, $B_m$, $C_m$, or $D_m$. By item (ii) in Lemma \ref{lem:BasicFactsAboutUalphaTalpha}, $F$ is the disjoint union of intervals each of which is closed on the left and open on the right. Furthermore, each of the end points of such intervals belongs to the set 
$$\Gamma=\{k\alpha\mod 1\,|\,k\in\{0,...,2\sum_{s=1}^{m}q_{n_s}-1\}\}\cup\{\frac{1}{2},1-\|q_{n_{(m-1)}}\alpha\|\}.$$
By condition (b) in Sublemma \ref{lem:BoundForTotalVarOfBoundedSets}, 
$$|\Gamma|=2\sum_{s=1}^{m}q_{n_s}+2=2\sum_{s=1}^{m-1}q_{n_s}+2+2q_{n_m}<q_{n_{(m-1)}+1}+3q_{n_m}<4q_{n_m},$$
where we used that $q_{n_m}\geq 2$. Thus, counting only the left end points of the intervals forming $F$, we see that $F$ is the disjoint union of at most $\frac{4q_{n_m}}{2}=2q_{n_m}$ intervals. 
Note now that 
$$\sup\{\frac{1}{\|x\|}\,|\,x\in F\}\leq\frac{1}{\|q_{n_{(m-1)}}\alpha\|}<2q_{n_{(m-1)}+1}.$$
Thus, $\text{Var}(\chi_F\frac{1}{\|x\|})\leq 8q_{n_{(m-1)}+1}q_{n_m}$, as claimed.
\end{proof}
\subsubsection{Proof of condition V2.}
\begin{proof}
For each $k\in\{0,...,q_{n_m}-1\}$, let 
$$E_{m,k}=(k\alpha\mod 1+\left[2\sum_{s=1}^{m-1}q_{n_s}\alpha\mod 1,(2\sum_{s=1}^{m-1}q_{n_s}+q_{n_m})\alpha\mod 1\right)).$$
By \eqref{eq:SymDiffOfU_m}, 
$$U_{\alpha,m}\triangle U_{\alpha,(m-1)}=\bigcup_{i=0}^{q_{n_m}-1}T_{\alpha,(m-1)}^i(J'_{\alpha,m})=\bigcup_{k=0}^{q_{n_m}-1}(E_{m,k}\times\Z_2),$$
and, hence, $E_m=\bigcup_{k=0}^{q_{n_m}-1}E_{m,k}$. Let $\{\lambda_1,...,\lambda_{q_{n_m}}\}\subseteq[0,1)$ and $\{\rho_1,...,\rho_{q_{n_m}}\}\subseteq[0,1)$ be enumerations in increasing order of the left  and  right end-points of the intervals $E_{m,0},...,E_{m,q_{n_m}-1}$, respectively. Since $m>1$, condition (b) in Sublemma \ref{lem:BoundForTotalVarOfBoundedSets} implies that $2\sum_{s=1}^{m-1}q_{n_s}+2q_{n_m}<q_{n_m+1}$, and, so,
\begin{equation}\label{eq:V2IntervalsAwayfromZero}
\|k\alpha\|\geq \|q_{n_m}\alpha\|>\frac{1}{2q_{n_m+1}},
\end{equation}
for $k\in\{2\sum_{s=1}^{m-1}q_{n_s},....,2\sum_{s=1}^mq_{n_s}-1\}$. Thus, for any $k\in\{1,...,q_{n_m}\}$, 
$$\|q_{n_m}\alpha\|\leq \lambda_k,\rho_k\leq 1-\|q_{n_m}\alpha\|,$$
which implies that $\|q_{n_m}\alpha\|\leq \lambda_k<\rho_k\leq 1-\|q_{n_m}\alpha\|$ (otherwise we would have $\|q_{n_m}\alpha\|=\|\rho_k-\lambda_k\|\geq 2\|q_{n_m}\alpha\|$). It now follows that
\begin{multline}\label{eq:V2TrivialBound}
\text{Var}(\chi_{E_m}\frac{1}{\|x\|})\leq \sum_{k=1}^{q_{n_m}}2\max\{\frac{1}{\|\lambda_k\|},\frac{1}{\|\rho_k\|}\}<\sum_{k=1}^{q_{n_m}}\frac{2}{\|\lambda_k\|}+\sum_{k=1}^{q_{n_m}}\frac{2}{\|\rho_k\|}.
\end{multline}
Observe now that 
any two of the points in $\{\lambda_1,...,\lambda_{q_{n_m}}\}$ and $\{\rho_1,...,\rho_{q_{n_m}}\}$, respectively,  are at least
\begin{equation}\label{eq:V2IntervalsAppart}
\|q_{n_m-1}\alpha\|>\frac{1}{2q_{n_m}}
\end{equation}
apart. Thus, by combining \eqref{eq:V2IntervalsAwayfromZero}, \eqref{eq:V2TrivialBound}, and \eqref{eq:V2IntervalsAppart}, we obtain 
\begin{multline*}
\text{Var}(\chi_{E_m}\frac{1}{\|x\|})<\sum_{k=1}^{q_{n_m}}\frac{2}{\|\lambda_k\|}+\sum_{k=1}^{q_{n_m}}\frac{2}{\|\rho_k\|}
\leq\sum_{k=0}^{q_{n_m}-1}\frac{4}{\|q_{n_m}\alpha\|+k\|q_{n_m-1}\alpha\|}\\
<\frac{4}{\|q_{n_m}\alpha\|}+\sum_{k=1}^{q_{n_m}-1}\frac{4}{k\|q_{n_m-1}\alpha\|}
<8q_{n_m+1}+8q_{n_m}\sum_{k=1}^{q_{n_m}-1}\frac{1}{k}<8(q_{n_m+1}+q_{n_m}+q_{n_m}\log(q_{n_m})).
\end{multline*}
\end{proof}
\subsubsection{Proof of condition V3.}
\begin{proof}
Let the intervals $E_{m,k}$, $k\in\{0,...,q_{n_m}-1\}$, and the points  $\{\lambda_1,...,\lambda_{q_{n_m}}\}$, and $\{\rho_1,...,\rho_{q_{n_m}}\}$ be as in the proof of condition V2 above. 
It follows that 
\begin{multline*}
\|q_{n_m}\alpha\|\sum_{k=1}^{q_{n_m}}\frac{1}{\rho_k}\leq \int_{E_m\subseteq [0,1)}\frac{1}{x}\text{d}\lambda'(x)
\leq \int_{E_m}\frac{1}{\|x\|}\text{d}\lambda'(x)
\leq\|q_{n_m}\alpha\|\sum_{k=1}^{q_{n_m}}\max\{\frac{1}{\|\lambda_k\|},\frac{1}{\|\rho_k\|}\}.
\end{multline*}
Arguing as in the proof of V2.,
\begin{multline*}
    \sum_{k=1}^{q_{n_m}}\max\{\frac{1}{\|\lambda_k\|},\frac{1}{\|\rho_k\|}\}
    <\sum_{k=0}^{q_{n_m}-1}\frac{2}{\|q_{n_m}\alpha\|+k\|q_{n_m-1}\alpha\|}\\
\leq \frac{2}{\|q_{n_m}\alpha\|}+\sum_{k=1}^{q_{n_m}-1}\frac{2}{k\|q_{n_m-1}\alpha\|}
\leq\frac{2}{\|q_{n_m}\alpha\|}+4q_{n_m}(1+\log(q_{n_m}))
\end{multline*}
Thus, by condition (a) in Sublemma \ref{lem:BoundForTotalVarOfBoundedSets},
\begin{multline*}
    \int_{E_m}\frac{1}{\|x\|}\text{d}\lambda'(x)
\leq\|q_{n_m}\alpha\|\sum_{k=1}^{q_{n_m}}\max\{\frac{1}{\|\lambda_k\|},\frac{1}{\|\rho_k\|}\}\\
\leq \|q_{n_m}\alpha\|(\frac{2}{\|q_{n_m}\alpha\|}+4q_{n_m}(1+\log(q_{n_m})))\leq 2+\frac{4q_{n_m}}{q_{n_m+1}}(1+\log(q_{n_m}))\\
\leq 2+\frac{4q_{n_m}}{Mq_{n_m}}(1+\log(q_{n_m}))=2+4\frac{1+\log(q_{n_m})}{M}
\end{multline*}
We now claim that for any $k\in\{0,...,q_{n_m}-1\}$, there exists $\ell,r\in\{0,...,q_{n_m}-1\}$ such that 
\begin{equation}\label{eq:PointsAreNearBy}
\|q_{n_m-1}\alpha\|\leq (k-\ell)\alpha\mod 1,(r-k)\alpha\mod 1\leq\|q_{n_m}\alpha\|+\|q_{n_m-1}\alpha\|.
\end{equation}
Indeed, take $\ell$ to be the unique element in 
$$\{0,...,q_{n_m}-1\}\cap\{k+q_{n_m-1},k+q_{n_m-1}-q_{n_m}\}$$
and $r$ the unique element in 
$$\{0,...,q_{n_m}-1\}\cap\{k-q_{n_m-1},k-q_{n_m-1}+q_{n_m}\},$$
Noting that $n_m$ is even (and $n_m-1$ is odd), \eqref{eq:PointsAreNearBy} follows immediately.\\
Thus, for any $k\in\{1,...,q_{n_m}-1\}$, 
$$\|\rho_{k+1}-\rho_k\|\leq \|q_{n_m}\alpha\|+\|q_{n_m-1}\alpha\|,$$
which in turn implies
$$\sum_{k=1}^{q_{n_m}}\frac{1}{\rho_k}\geq \sum_{k=0}^{q_{n_m}-1}\frac{1}{\rho_1+k(\|q_{n_m}\alpha\|+\|q_{n_m-1}\alpha\|)}.$$
By condition (b) in Sublemma  \ref{lem:BoundForTotalVarOfBoundedSets}, 
$$2\sum_{s=1}^{m-1}q_{n_s}<q_{n_{(m-1)}+1}< q_{n_m}<2q_{n_m}<2\sum_{s=1}^mq_{n_s}<q_{n_m+1}.$$
So, since $n_m$ is even, 
\begin{equation*}\label{eq:FirstLeftEndPoint}
\lambda_1=\|q_{n_m}\alpha\|\text{ and }\rho_1=2\|q_{n_m}\alpha\|,
\end{equation*}
We now have
 \begin{multline*}
 \int_{E_m}\frac{1}{\|x\|}\text{d}\lambda'(x)\geq \int_{E_m}\frac{1}{x}\text{d}\lambda'(x)\geq
 \|q_{n_m}\alpha\| \sum_{k=1}^{q_{n_m}}\frac{1}{\rho_k}\\
 \geq  \|q_{n_m}\alpha\|\sum_{k=0}^{q_{n_m}-1}\frac{1}{2\|q_{n_m}\alpha\|+k(\|q_{n_m}\alpha\|+\|q_{n_m-1}\alpha\|)}
\geq \frac{\|q_{n_m}\alpha\|}{2\|q_{n_m-1}\alpha\|}\sum_{k=1}^{q_{n_m}}\frac{1}{k}\\
 \geq\frac{q_{n_m}}{2(q_{n_m+1}+q_{n_m})}\sum_{k=1}^{q_{n_m}}\frac{1}{k}>\frac{q_{n_m}\log(q_{n_m})}{2(q_{n_m+1}+q_{n_m})}\geq \frac{q_{n_m}\log(q_{n_m})}{2(M+2)q_{n_m}}=\frac{\log(q_{n_m})}{2(M+2)}.
 \end{multline*}
 We are done. 
\end{proof}
\section{Unique ergodicity of the flow}\label{sec:ue}
In Sections $5-7$ we showed that the flow $T^f$ described on Theorem \hyperlink{theorem:Main2'}{A$^\prime$} is not strongly mixing. In this section we show that the IET underlying $T^f$, $T=T_{\alpha_0,n_k}$, is uniquely ergodic. To do this, we will closely follow the ideas in the proof of unique ergodicity of a related IET, which we call $S$, in \cite{chaika}. We remark that since many of the properties of such an $S$ are incompatible with the properties of $T$, we will need to adapt some of the arguments used in \cite{chaika}.\\

Let $\alpha=\alpha_0$ be the irrational number guaranteed to exist by Proposition \hyperlink{prop:PropertiesOFAlpha0}{B}. So, in particular, there exists an increasing sequence $(n_k)_{k\in\N}$ in $2\N$ such that if we denote by $[0;a_1,...]$ the continued fraction expansion of $\alpha$ and by $(q_n)_{n\in\N}$ its denominators sequence, we have 
\begin{enumerate}
    \item [($\alpha$.5)] For every $k\in\N$, $a_{n_{(2k-1)}+1}= 3$.
    \item [($\alpha$.6)]  $\lim_{k\rightarrow\infty}(\sum_{s=1}^kq_{n_s})/q_{n_{(k+1)}}=0$.
    \item [($\alpha$.7)] $\lim_{k\rightarrow\infty}a_{n_{2k}+1}=\infty$. 
    \item [($\alpha$.8)] $\lim_{k\rightarrow\infty}(q_{n_k}\sum_{s=k+1}^\infty\|q_{n_s}\alpha_0\|)=0.$
\end{enumerate}
As we will see, the following conditions, which are weaker than ($\alpha$.5)-($\alpha$.8) and we formulate for $\alpha=\alpha_0$, already imply the unique ergodicity of $T_{\alpha_0,n_k}$.
\begin{enumerate}
    \item [(E.1)] There exists an $M\in\N$ such that for every $k\in\N$, $3\leq a_{n_{(2k-1)}+1}\leq M$.
    \item [(E.2)]  $\lim_{k\rightarrow\infty}(\sum_{s=1}^{2k-1}q_{n_s})/q_{n_{2k}}=0$.
    \item [(E.3)] $\lim_{k\rightarrow\infty}a_{n_{2k}+1}=\infty$. 
    \item [(E.4)] $\lim_{k\rightarrow\infty}(q_{n_k}\sum_{s=k+1}^\infty\|q_{n_s}\alpha\|)=0.$
\end{enumerate}
Our goal in this section is to prove the following result. 
\begin{proposition}\label{thm:UniqueErgodicity}
    Let $\alpha\in (0,1)$ be an irrational number and let $(n_k)_{k\in\N}$ be an increasing sequence in $2\N$. Suppose that $\alpha$ and $(n_k)_{k\in\N}$ satisfy conditions (E.1)-(E.4). Then $T=T_{\alpha,n_k}$, as defined in \eqref{eq:skewAlpha},  is uniquely ergodic. 
\end{proposition}
The proof of Proposition \ref{thm:UniqueErgodicity} will require various results which we divide into three groups. The first group of results deals with basic measure theoretical estimates involving the sets of the form $U_{\alpha,k}$ defined in \eqref{eq:DefnU_m}. The second, deals with the potential ergodic measures for $T_{\alpha,n_k}$. The third group deals with the "asymptotic" equidistribution  of $(T_{\alpha,2k})_{k\in\N}$. For the rest of this section we will fix $\alpha$ and $(n_k)_{k\in\N}$ satisfying conditions (E.1)-(E.4).
\subsection{Basic measure theoretical estimates}
For each $k\in\N$, let $T_k=T_{\alpha,k}$ be as defined in \eqref{eq:SthSkew}, let $U_k=U_{\alpha,k}$ be as defined in \eqref{eq:DefnU_m}, and let $V_k=V_{\alpha,k}$ be as defined in \eqref{eq:DefnV_m}. For each $k\in\N$ let 
$$J_k:=J'_{\alpha,k}=[2\sum_{s=0}^{k-1}q_{n_s}\alpha,2\sum_{s=0}^{k-1}q_{n_s}\alpha+q_{n_k}\alpha)\times\Z_2,$$
where, by convention, $q_{n_0}=0$. Observe that, by  \eqref{eq:BoundOnSumOfClosestIntegers} and (E.1), we have 
$$2\sum_{s=1}^\infty\|q_{n_s}\alpha\|<\frac{4}{q_{n_1+1}}\leq \frac{4}{3q_{2}+q_1}\leq \frac{4}{7}<1.$$
So, by Lemma  \ref{lem:BasicFactsAboutUalphaTalpha} item (i), $U_k$ and $V_k$ are $T_k$-invariant for each $k\in\N$.\\
By \eqref{eq:SymDiffOfU_m}, for each $k\in\N$ we have 
\begin{equation}\label{eq:SymDiffU_kSimpleNotation}
U_k\triangle U_{k-1}=\bigcup_{i=0}^{q_{n_k}-1}T_{(k-1)}^i(J_k),
\end{equation}
where $U_0=\T\times\{0\}$ and for $(x,j)\in\T\times\Z_2$, $T_0(x,j)=(x+\alpha,j)$.\\
Note that the sets $T^iJ_k$, $i=0,...,q_{n_k}-1$, are pairwise disjoint. Thus, 
$$\frac{1}{a_{n_k+1}+2}<\frac{q_{n_k}}{q_{n_k}+q_{n_k+1}}<\lambda(U_k\triangle U_{k-1})=q_{n_k}\|q_{n_k}\alpha\|<\frac{q_{n_k}}{q_{n_k+1}}<\frac{1}{a_{n_k+1}}.$$
It now follows  from (E.1) that
\begin{equation}\label{eq:SymDiffU_2kU_2k+1}
   \frac{1}{M+2}\leq \frac{1}{a_{n_{(2k+1)}+1}+2} <\lambda(U_{2k+1}\triangle U_{2k})<\frac{1}{a_{n_{(2k+1)}+1}}\leq \frac{1}{3}
\end{equation}
for any $k\in\N$ and from (E.3) that 
\begin{equation}\label{eq:EvenSymDiffGoesToZero}
    \lim_{k\rightarrow\infty}\lambda(U_{2k}\triangle U_{2k-1})\leq \lim_{k\rightarrow\infty}\frac{1}{a_{n_{2k}+1}}=0.
\end{equation}
\subsection{Potential ergodic measures for $T_{\alpha,n_k}$}
We define the involution $\textit{i}:\mathbb T\times\Z_2\rightarrow\mathbb T\times\Z_2$ by $\textit{i}(x,j)=(x,j+1)$. It was shown in \cite{chaika} that for every $k\in\N\cup\{0\}$, $T_k\circ \textit{i}=\textit{i}\circ T_k$, $\textit{i}(U_k)=V_k$, and, hence, \\
 $$\lambda(U_k)=\lambda(V_k)=1/2.$$
 (Note that here we are using the fact that $\{U_k,V_k\}$ is a partition of $\T\times\Z_2$.) We will also let $R=R_\alpha:\T\rightarrow\T$ be defined by $R(x)=x+\alpha$.\\
 
 The following result shows that there are at most two $T$-invariant  ergodic probability measures. 
\begin{lemma}[Cf. Lemma 7.5 in \cite{chaika}]\label{lem:NonUniqueErgForTxZ_2}
If $T$ is not uniquely ergodic, then there exist exactly two $T$-invariant ergodic probability measures $\mu$ and $\nu$. Furthermore, $\mu$ and $\nu$ are such that (a)  $\lambda=\frac{1}{2}(\mu+\nu)$ and (b) $\nu=\mu\circ \textit{i}^{-1}$, where $\lambda$ denotes the normalized Lebesgue measure on $\mathbb T\times\Z_2$.
\end{lemma}
\begin{proof}
Let $\mu$ be an ergodic probability measure for $T$. Since $\textit{i}\circ T=T\circ\textit{i}$, the probability measure $\nu=\mu\circ\textit{i}^{-1}$ is also $T$-invariant and ergodic. Let $\sigma=\frac{1}{2}(\mu+\nu)$. Let $V$ be a (possibly empty)  open interval in $\mathbbm T$ and let $U=V\times\{0\}$. On one hand, since $R=R_\alpha$ is an irrational rotation, 
$$2\lambda(U)=\lim_{N\rightarrow\infty}\frac{1}{N}\sum_{j=0}^{N-1}\chi_V(R^jx),$$
for every $x\in\mathbb T$. 
On the other hand, by  Birkhoff's Ergodic Theorem, for $\mu$-almost every $(x,j)\in\mathbb T\times \Z_2$, 
$$\sigma(U)=\frac{1}{2}(\mu(U)+\nu(U))=\frac{1}{2}(\mu(U)+\mu(\textit{i}^{-1}U))=\lim_{N\rightarrow\infty}\frac{1}{2N}\sum_{j=0}^{N-1}\chi_{U\cup\textit{i}^{-1}U}(T^j(x,j)).$$
Noting that 
$$\lim_{N\rightarrow\infty}\frac{1}{2N}\sum_{j=0}^{N-1}\chi_V(R^jx)=\lim_{N\rightarrow\infty}\frac{1}{2N}\sum_{j=0}^{N-1}\chi_{U\cup\textit{i}^{-1}U}(T^j(x,j))$$
for $\mu$-almost every $(x,j)\in\mathbbm T\times\Z_2$, we see that $\sigma(U)=\lambda(U)$. In a similar way, one can show that $\sigma(V\times\{1\})=\lambda(V\times\{1\})$. Thus, by the $\pi$-$\lambda$ Theorem, $\sigma=\lambda$ and, hence,  $\lambda=\frac{1}{2}(\mu+\nu)$.\\
Since the ergodic measure $\mu$ in the previous argument was arbitrary, we see that for any two ergodic measures $\mu$, $\mu'$, 
$$\lambda=\frac{1}{2}(\mu+\mu\circ\textit{i}^{-1})=\frac{1}{2}(\mu'+\mu'\circ\textit{i}^{-1}).$$
Thus, by the uniqueness of the ergodic decomposition for $\lambda$, either $\mu=\mu'$ or $\mu=\mu'\circ\textit{i}^{-1}$. We are done. 
\end{proof}
\subsection{Asymptotic equidistribution of $(T_{\alpha,2k})_{k\in\N}$}
In this subsection we prove the following lemma. 
\begin{lemma}[Cf. Lemma 7.7 in \cite{chaika}] \label{lem:GoodErgodicSumsOnU_2k}
Let $A\subseteq \mathbb T\times \Z_2$ be measurable. For any $\epsilon>0$,
\begin{equation}\label{eq:ConvergenceToHalfLebesgue1}
\lim_{k\rightarrow\infty}\lambda\left(\{(x,j)\in U_{2k}\,|\,\left|\frac{1}{q_{n_{2k}}}\sum_{i=0}^{q_{n_{2k}}-1}\chi_A(T^i_{2k}(x,j))-2\lambda(U_{2k}\cap A)\right|>\epsilon\}\right)=0.
\end{equation}
A similar result holds when one replaces $U_{2k}$ by $V_{2k}$.
\end{lemma}
\begin{proof}
By \eqref{eq:EvenSymDiffGoesToZero}, we have that 
\begin{equation}\label{eq:SmallAsymSymDiff}
\lim_{k\rightarrow\infty}\lambda(U_{2k}\triangle U_{2k-1})=0.
\end{equation}
Thus, if $\epsilon>0$ and 
$$\lim_{k\rightarrow\infty}\lambda\left(\{(x,j)\in U_{2k}\,|\,\left|\frac{1}{q_{n_{2k}}}\sum_{i=0}^{q_{n_{2k}}-1}\chi_A(T^i_{2k}(x,j))-2\lambda(U_{2k-1}\cap A)\right|>\frac{\epsilon}{2}\}\right)=0,$$
we have that \eqref{eq:ConvergenceToHalfLebesgue1} holds. Thus, in order to prove Lemma \ref{lem:GoodErgodicSumsOnU_2k}, it suffices to show that 
\begin{equation}\label{eq:ConvergenceToHalfLebesgue1.5}
\lim_{k\rightarrow\infty}\lambda\left(\{(x,j)\in U_{2k}\,|\,\left|\frac{1}{q_{n_{2k}}}\sum_{i=0}^{q_{n_{2k}}-1}\chi_A(T^i_{2k}(x,j))-2\lambda(U_{2k-1}\cap A)\right|>\epsilon\}\right)=0
\end{equation}
for any $\epsilon>0$.\\
Fix now $\epsilon>0$ and note that, by \eqref{eq:SmallAsymSymDiff}, in order to prove \eqref{eq:ConvergenceToHalfLebesgue1.5}, all we need to show is that 
\begin{equation}\label{eq:ConvergenceToHalfLebesgue2}
\lim_{k\rightarrow\infty}\lambda\left(\{(x,j)\in U_{2k-1}\,|\,\left|\frac{1}{q_{n_{2k}}}\sum_{i=0}^{q_{n_{2k}}-1}\chi_A(T^i_{2k}(x,j))-2\lambda(U_{2k-1}\cap A)\right|>\epsilon\}\right)=0.
\end{equation}
Note now that for any $k\in\N$,
$$\lambda(\bigcup_{i=0}^{q_{n_{2k}}-1}\{(x,j)\in \mathbb T\times \Z_2\,|\, T^i_{2k}(x,j)\neq T^i_{2k-1}(x,j)\})\leq 2q_{n_{2k}}\|q_{n_{2k}}\alpha\|\leq \frac{2}{a_{n_{2k}+1}}.$$
So, by (E.3), in order to prove  \eqref{eq:ConvergenceToHalfLebesgue2} it suffices to show that 
\begin{multline}\label{eq:ConvergenceToHalfLebesgue3}
\lim_{k\rightarrow\infty}\lambda\left(\{(x,j)\in U_{2k-1}\,|\,\left|\frac{1}{q_{n_{2k}}}\sum_{i=0}^{q_{n_{2k}}-1}\chi_A(T^i_{2k-1}(x,j))-2\lambda(U_{2k-1}\cap A)\right|>\epsilon\}\right)\\
=0.
\end{multline}
Suppose now that $A$ is an interval in $\mathbb T\times \Z_2$ and let $\tilde A_k$ be the projection of $A\cap U_{2k-1}$ into $\mathbb T$. In other words, 
$$\tilde A_k:=\{x\in\T\,|\,(x,j)\in A\cap U_{2k-1}\}.$$ 
Observe that, by Lemma \ref{lem:BasicFactsAboutUalphaTalpha} item (iii),  for each $x\in\T$ there is a unique $j\in\Z_2$ with $(x,j)\in U_{2k-1}$. Thus, since $U_{2k-1}$ is $T_{2k-1}$-invariant, we have that for any $(x,j)\in U_{2k-1}$,
$$\frac{1}{q_{n_{2k}}}\sum_{i=0}^{q_{n_{2k}}-1}\chi_A(T^i_{2k-1}(x,j))=\frac{1}{q_{n_{2k}}}\sum_{i=0}^{q_{n_{2k}}-1}\chi_{A\cap U_{2k-1}}(T^i_{2k-1}(x,j))=\frac{1}{q_{n_{2k}}}\sum_{i=0}^{q_{n_{2k}}-1}\chi_{\tilde A_k}(R^ix).$$
By Lemma \ref{lem:BasicFactsAboutUalphaTalpha} item (ii), $U_{2k-1}$ is the disjoint union of at most $2\sum_{s=0}^{2k-1}q_{n_s}$ disjoint intervals. It now follows from  condition (E.2) that $\tilde A_k$ is the disjoint union of $o(q_{n_{2k}})$  intervals. Thus, by Denjoy-Koksma inequality, for any $x\in\T$, 
$$|\frac{1}{q_{n_{2k}}}\sum_{i=0}^{q_{n_{2k}}-1}\chi_{\tilde A_k}(R^ix)-\lambda'(\tilde A_k)|\leq \frac{\text{Var}(\chi_{\tilde A_k})}{q_{n_{2k}}}=\frac{o(q_{n_{2k}})}{q_{n_{2k}}},$$
where $\lambda'$ denotes the normalized Lebesgue measure on $\mathbb T$. Noting that 
$$\lambda'(\tilde A_k)=2\lambda(U_{2k-1}\cap A),$$
we see that when  $A$ is an interval, \eqref{eq:ConvergenceToHalfLebesgue3} holds for any given $\epsilon>0$.\\
We now assume that $A$ is arbitrary. Fix $\delta>0$ and let $A_1,...,A_N\subseteq \mathbb T\times \Z_2$ be disjoint intervals such that $\lambda(A\triangle \bigcup_{j=1}^NA_j)<\delta$. It follows that 
$$\frac{1}{q_{n_{2k}}}\sum_{i=0}^{q_{n_{2k}}-1}\int_{\mathbb T\times \Z_2}\Big|\chi_A(T^i_{2k-1}(x,j))-\chi_{\bigcup_{s=1}^NA_s}(T^i_{2k-1}(x,j))\Big|\text{d}\lambda(x,j)<\delta$$
and hence 
$$\lambda(\{(x,j)\in\mathbb T\times \Z_2\,|\,|\frac{1}{q_{n_{2k}}}\sum_{i=0}^{q_{n_{2k}}-1}\left(\chi_A(T^i_{2k-1}(x,j))-\chi_{\bigcup_{s=1}^NA_s}(T^i_{2k-1}(x,j))\right)|> \sqrt{\delta}\})\leq \sqrt{\delta}.$$
Picking $\delta\in (0,1)$ such that $\delta<\sqrt{\delta}<\frac{\epsilon}{4}$, we obtain,
\begin{multline*}
    \lim_{k\rightarrow\infty}\lambda\left(\{(x,j)\in U_{2k-1}\,|\,\left|\frac{1}{q_{n_{2k}}}\sum_{i=0}^{q_{n_{2k}}-1}\chi_A(T^i_{2k-1}(x,j))-2\lambda(U_{2k-1}\cap A)\right|>\epsilon\}\right)\\
    \leq \sqrt{\delta}\\
    + \lim_{k\rightarrow\infty}\lambda\left(\{(x,j)\in U_{2k-1}\,|\,\left|\frac{1}{q_{n_{2k}}}\sum_{i=0}^{q_{n_{2k}}-1}\chi_{\bigcup_{s=1}^NA_s}(T^i_{2k-1}(x,j))-2\lambda(U_{2k-1}\cap \bigcup_{s=1}^N A_s)\right|>\frac{\epsilon}{4}\}\right)\\
    \leq \sqrt{\delta}\\
    + \lim_{k\rightarrow\infty}\lambda\left(\{(x,j)\in U_{2k-1}\,|\,\sum_{s=1}^N\left|\frac{1}{q_{n_{2k}}}\sum_{i=0}^{q_{n_{2k}}-1}\chi_{A_s}(T^i_{2k-1}(x,j))-2\lambda(U_{2k-1}\cap A_s)\right|>\frac{\epsilon}{4}\}\right)\\
    \leq \sqrt{\delta}\\
    +\lim_{k\rightarrow\infty}\sum_{s=1}^N\lambda\left(\{(x,j)\in U_{2k-1}\,|\,\left|\frac{1}{q_{n_{2k}}}\sum_{i=0}^{q_{n_{2k}}-1}\chi_{A_s}(T^i_{2k-1}(x,j))-2\lambda(U_{2k-1}\cap A_s)\right|>\frac{\epsilon}{4N}\}\right)\\
    =\sqrt{\delta}.
\end{multline*}
Taking $\delta$ arbitrarily small, we complete the proof. 
\end{proof}
\subsection{The proof of Proposition \ref{thm:UniqueErgodicity}}
\begin{proof}[Proof of Proposition \ref{thm:UniqueErgodicity}]
Suppose for the sake of contradiction that $T$ is not uniquely ergodic. Then, by Lemma \ref{lem:NonUniqueErgForTxZ_2},  there exist two ergodic probability measures $\mu$, $\nu$ and a measurable set $A\subseteq \mathbb T\times \Z_2$ such that $\mu(A)=1$, $\nu(A)=0$, $\lambda(A)=\frac{1}{2}$, and $TA=A$ up to a set of $\lambda$- (and, hence, $\mu$- and $\nu$-) measure zero. We claim that for any $\epsilon\in(0,1/2)$, there exists a $k_\epsilon\in\N$ such that if $k>k_\epsilon$, then 
\begin{equation}\label{eq:ExtremeValuesOfAcapUk}
2\lambda(A\cap U_{k})\in [0,\epsilon]\cup [1-\epsilon,1].
\end{equation}
Indeed, by (E.4), 
$$\lim_{k\rightarrow\infty}\lambda(\bigcup_{s=0}^{q_{n_{2k}}-1}\{(x,j)\in\mathbb T\times\Z_2\,|\,T^s(x,j)\neq T^s_{2k}(x,j)\})=\lim_{k\rightarrow\infty}q_{n_{2k}}\sum_{s=2k+1}^\infty\|q_{n_s}\alpha\|=0.$$
Thus, by Lemma \ref{lem:GoodErgodicSumsOnU_2k} and the $T$-invariance of $A$,
\begin{multline*}
\frac{1}{2}=\lim_{k\rightarrow\infty}\lambda\left(\{(x,j)\in U_{2k}\,|\,\left|\frac{1}{q_{n_{2k}}}\sum_{i=0}^{q_{n_{2k}}-1}\chi_A(T^i(x,j))-2\lambda(U_{2k}\cap A)\right|\leq\epsilon\}\right)\\
=\lim_{k\rightarrow\infty}\lambda\left(\{(x,j)\in U_{2k}\,|\,\left|\chi_A(x,j)-2\lambda(U_{2k}\cap A)\right|\leq\epsilon\}\right)\\
=\lim_{k\rightarrow\infty}\Big(\lambda(\{(x,j)\in U_{2k}
\cap A\,|\,\left|1-2\lambda(U_{2k}\cap A)\right|\leq\epsilon\})\\
+\lambda(\{(x,j)\in U_{2k}\cap A^c\,|\,\left|2\lambda(U_{2k}\cap A)\right|\leq\epsilon\})\Big),
\end{multline*}
where $A^c=(\mathbb T\times \Z_2)\setminus A$. Note that, since $\epsilon\in (0,1/2)$, for each $k\in\N$, at  most one of $|1-2\lambda(U_{2k}\cap A)|\leq \epsilon$ and $|2\lambda(U_{2k}\cap A)|\leq \epsilon$ can hold. Also note that if neither of $|1-2\lambda(U_{2k}\cap A)|\leq \epsilon$ and $|2\lambda(U_{2k}\cap A)|\leq \epsilon$ holds, then 
\begin{multline*}
\lambda(\{(x,j)\in U_{2k}
\cap A\,|\,\left|1-2\lambda(U_{2k}\cap A)\right|\leq \epsilon\})\\
+\lambda(\{(x,j)\in U_{2k}\cap A^c\,|\,\left|2\lambda(U_{2k}\cap A)\right|\leq \epsilon\})=0.
\end{multline*}
Thus, for $k\in\N$ large enough, we have that exactly one of $|1-2\lambda(U_{2k}\cap A)|\leq \epsilon$ and $|2\lambda(U_{2k}\cap A)|\leq \epsilon$ holds. Thus, since $\lim_{k\rightarrow\infty}\lambda(U_{2k}\triangle U_{2k-1})=0$,  we see that \eqref{eq:ExtremeValuesOfAcapUk} holds.\\

To complete the proof let $\epsilon=\frac{1}{2(M+3)}\leq\frac{1}{12}$ and let $k>k_{\epsilon}$. By \eqref{eq:ExtremeValuesOfAcapUk}, we have 
$$2\lambda(U_{2k}\cap A),2\lambda(U_{2k+1}\cap A)\in [0,\epsilon]\cup [1-\epsilon,1].$$
If $2\lambda(U_{2k}\cap A)\geq 1-\epsilon$ and $2\lambda(U_{2k+1}\cap A^c)\geq 1-\epsilon$, then, by \eqref{eq:SymDiffU_2kU_2k+1},
$$1=\lambda(A\triangle A^c)\leq\lambda(A\triangle U_{2k})+\lambda(U_{2k}\triangle U_{2k+1})+\lambda(A^c\triangle U_{2k+1})\leq 2\epsilon+\frac{1}{3}\leq \frac{1}{2}<1.$$
Thus, we cannot have that  $2\lambda(U_{2k}\cap A)\geq 1-\epsilon$ and $2\lambda(U_{2k+1}\cap A^c)\geq 1-\epsilon$. 
Similarly, we cannot have that $2\lambda(U_{2k}\cap A^c)\geq 1-\epsilon$ and $2\lambda(U_{2k+1}\cap A)\geq 1-\epsilon$.\\
If $2\lambda(U_{2k}\cap A^c)\geq 1-\epsilon$ and $2\lambda(U_{2k+1}\cap A^c)\geq 1-\epsilon$, 
then, by \eqref{eq:SymDiffU_2kU_2k+1},
$$\frac{1}{M+2}<\lambda(U_{2k}\triangle U_{2k+1})\leq \lambda(U_{2k}\triangle A^c)+\lambda(U_{2k+1}\triangle A^c)\leq 2\epsilon=\frac{1}{M+3}.$$
This proves that $2\lambda(U_{2k}\cap A^c)\geq 1-\epsilon$ and $2\lambda(U_{2k+1}\cap A^c)\geq 1-\epsilon$ cannot hold.  
Similarly,    $2\lambda(U_{2k}\cap A)\geq 1-\epsilon$ and $2\lambda(U_{2k+1}\cap A)\geq 1-\epsilon$ is impassible. Since in every one  of these four cases we reach a contradiction, we must have that $T$ is uniquely ergodic. 
\end{proof}


\end{document}